\documentclass[a4paper, 11pt, english]{article}
\usepackage[utf8]{inputenc}
\usepackage[T1]{fontenc}
\usepackage{babel}

\title{Quantitative propagation of smallness and spectral estimates  for the Schrödinger operator}
\author{
Kévin Le Balc'h\footnote{Inria, Sorbonne Université, Université de Paris, CNRS, Laboratoire Jacques-Louis Lions, Paris, France. {\tt kevin.le-balc-h@inria.fr}}
\and
Jérémy Martin \footnote{Inria, Sorbonne Université, Université de Paris, CNRS, Laboratoire Jacques-Louis Lions, Paris, France. {\tt jeremy.a.martin@inria.fr.}
}}
\date{\today}

\usepackage[top=3cm, bottom=2cm, left=3cm, right=3cm]{geometry}	

\usepackage{enumitem}
\usepackage{array}

\usepackage{amsfonts}
\usepackage{amsmath}
\usepackage{amsthm}
\usepackage{amssymb}
\usepackage{mathtools}

\usepackage{multicol}

\usepackage{bbm}

\usepackage{stmaryrd}

\usepackage[scr]{rsfso}

\usepackage{comment}

\usepackage[dvipsnames]{xcolor}

\usepackage{hyperref}

\usepackage{cleveref}

\Crefname{paragraph}{Section}{Sections}

\hypersetup{	
  colorlinks=true,
  breaklinks=true,
  urlcolor= red,
  linkcolor= red,
  citecolor=ForestGreen
}

\numberwithin{equation}{section}

\newtheorem{theorem}{Theorem}[section]
\newtheorem{definition}[theorem]{Definition}
\newtheorem{lemma}[theorem]{Lemma}
\newtheorem{proposition}[theorem]{Proposition}

\numberwithin{equation}{section}

\newcommand{\norme}[1]{\left\lVert#1\right\rVert}

\newcommand{\ensemblenombre}[1]{\mathbb{#1}}

\newcommand{\Z}{\ensemblenombre{Z}}

\newcommand{\R}{} 
\renewcommand{\R}{\ensemblenombre{R}}

\newcommand{\rr}{\mathbb{R}}
\newcommand{\nn}{\mathbb{N}}

\def\un{{\mathrm{1~\hspace{-1.4ex}l}}}

\newcommand{\divergence}{\mathrm{div}}

\newcommand\bna{\begin{eqnarray*}}
\newcommand\ena{\end{eqnarray*}}

\newcommand\bnan{\begin{eqnarray}}
\newcommand\enan{\end{eqnarray}}

\begin{document}

\maketitle

\begin{abstract}
In this paper, we investigate quantitative propagation of smallness properties for the Schrödinger operator on a bounded domain in $\R^d$. We extend Logunov, Malinnikova's results concerning propagation of smallness for $A$-harmonic functions to solutions of divergence elliptic equations perturbed by a bounded zero order term. We also prove similar results for gradient of solutions to some particular equations. This latter result enables us to follow the recent strategy of Burq, Moyano for the obtaining of spectral estimates on rough sets for the Schrödinger operator. Applications to observability estimates and to the null-controllability of associated parabolic equations posed on compact manifolds or the whole euclidean space are then considered.
\end{abstract}



 \small
\tableofcontents
\normalsize

\section{Introduction}

This paper contains several quantitative results on propagation of smallness for solutions of elliptic partial differential equations and their applications to spectral estimates.

More precisely, our attempt is to derive three spheres theorems for solutions and their gradients to divergence elliptic equations perturbed by a bounded zero order term, i.e. Schrödinger type equations
\begin{equation}
\label{eq:schrodingerequationIntro}
-\divergence(A(x) \nabla u)+ V(x) u = 0,\qquad x \in \Omega,
\end{equation}
where $\Omega$ is a smooth bounded open connected set of $\R^d$, with $d\geq 1$, $A=A(x)$ is a symmetric uniformly elliptic matrix with Lipschitz entries and $V=V(x)$ is a bounded real-valued function. From Carleman estimates, it is by now classical that the following interpolation inequality holds. For $B \subset \mathcal K \subset \subset \Omega$ with $B$ open and $\mathcal K $ compact, there exist $C>0$ and $\alpha \in (0,1)$ such that for every solution $u$ to \eqref{eq:schrodingerequationIntro}, we have
\begin{equation}
    \label{eq:threesphereB}
    \sup_{\mathcal K } |u| \leq C (\sup_{B} |u|)^{\alpha} (\sup_{\Omega} |u|)^{1-\alpha}.
\end{equation}
See for instance \cite[Theorem 5.1]{LRL12} or \cite[Corollary 2.3]{LM20} and the references therein.

In \cite{LM18}, in relation to the applications of the Yau's conjecture on the volume of the nodal sets for Laplace eigenfunctions (see \cite{Log18a} and \cite{Log18b}), generalization of the three spheres theorems \eqref{eq:threesphereB} for wild sets for solutions $u$ to 
\begin{equation}
\label{eq:schrodingerequationV0Intro}
-\divergence(A(x) \nabla u) = 0,\qquad x \in \Omega,
\end{equation}
was considered. More precisely, Logunov and Malinnikova obtain in \cite[Theorem 2.1]{LM18} that, given $E \subset \mathcal K \subset \subset \Omega$ with $E$ of positive $d$-dimensional Lebesgue measure, there exist $C>0$ and $\alpha \in (0,1)$ such that for every solution $u$ to \eqref{eq:schrodingerequationV0Intro}, we have
\begin{equation}
    \label{eq:threesphereBWild}
    \sup_{\mathcal{K}} |u| \leq C (\sup_{E} |u|)^{\alpha} (\sup_{\Omega} |u|)^{1-\alpha}.
\end{equation}
One can even assume that $E$ has positive $(d-1+\delta)$-Hausdorff
content for every $\delta >0$. Note that this latter result is sharp in the sense that zeros of harmonic functions in $\R^d$ for $d \geq 2$ may have positive $(d-1)$-Hausdorff content. The propagation of smallness for gradients from sets of positive $(d-1-\delta)$-Hausdorff content for some (small) $\delta$ was also obtained in \cite[Theorem 5.1]{LM18}. As the zeros of $|\nabla u|$ was shown in \cite{NV17} to have finite $(d-2)$-Hausdorff measure, it was conjectured in \cite{LM18} that the result for $|\nabla u|$ should be expected to hold from sets of positive $(d-2+\delta)$-Hausdorff content for any $\delta >0$. Up to now, this conjecture is still open.

The first goal of this paper is then to extend Logunov, Malinnikova's results to the Schrödinger type equation \eqref{eq:schrodingerequationIntro}. Propagation of smallness for solutions are obtained in full generality in the same setting of \cite{LM18}. On the other hand, propagation of smallness for gradients are only derived in a particular setting. Indeed, one cannot expect to derive propagation of smallness for gradients of solutions to \eqref{eq:schrodingerequationIntro} in full generality because as noted in \cite[Remark p. 362]{HHOHON99}, every closed subset of $\R^d$ could be the critical set of such a function so there is no hope to propagate smallness from sets of $(d-1-\delta)$-Hausdorff contents, even for small $\delta>0$. Nevertheless, our particular result is sufficient for the applications to spectral estimates that we next describe.\medskip

Let $M$ be a compact connected Riemannian manifold of dimension $d$, possibly with boundary, equipped with a Riemannian metric $g$. We assume that $M$ is a $\mathcal C^1 \cap W^{2,\infty}$, in the sense that the changes of charts are $\mathcal C^1$ with Lipschitz derivatives. We consider the following associated elliptic operator
\begin{equation}
\label{eq:defHavM}
    H_{g, V} u =-\Delta_g u + V(x) u,\qquad x \in M,
\end{equation}
where $g$ is assumed to be a Lipschitz positive definite metric, $\Delta_g= \divergence_g \circ \nabla_g$ is the Laplace-Beltrami operator and $V=V(x)$ is a bounded real-valued function. The operator $H_{g, V}$ eventually completed with Dirichlet or Neumann boundary conditions, is an unbounded self-adjoint with compact resolvent operator in $L^2(M)$. Consequently, it admits an orthonormal basis in $L^2(M)$ of eigenfunctions denoted by $(\phi_k)_{k \geq 1}$, associated to the sequence of real eigenvalues $(\lambda_k)_{k \geq 1}$ which satisfies $\lambda_k \underset{k \to +\infty}{\longrightarrow}+\infty$ and $\lambda_k \geq -\|V\|_{L^{\infty}}$, for all $k \in \nn$. Given $\Lambda >0$, we introduce the spectral projector $\Pi_{\Lambda}$ as follows
\begin{equation}
    \Pi_{\Lambda} u = \sum_{\lambda_k \leq \Lambda} \langle u, \phi_k\rangle_{L^2} \phi_k\qquad \forall u   \in L^2(M).
\end{equation}
Given a nonempty open subset $\omega$ of $M$, Jerison and Lebeau obtain in \cite{JL99} through Carleman estimates the following spectral inequality
\begin{equation}
\label{eq:PiLambdaSpectralIntro}
\norme{\Pi_{\Lambda} u}_{L^2(M)} \leq C e^{C \sqrt{\Lambda}} \norme{\Pi_{\Lambda} u}_{L^2(\omega)}\qquad \forall u \in L^2(M).
\end{equation}
This type of estimate \eqref{eq:PiLambdaSpectralIntro} is a generalization to linear combination of eigenfunctions of the well-known doubling inequality of Donnelly, Fefferman \cite{DF88} valid for one eigenfunction. Moreover, \eqref{eq:PiLambdaSpectralIntro} combining with the so-called Lebeau-Robbiano method \cite{LR95} leads to the small-time null-controllability of the associated parabolic equations with a control localized in $\omega$. In \cite{AEWZ14} (see also \cite{AE13}), Apraiz, Escauriaza, Wang and Zhang generalize  \eqref{eq:PiLambdaSpectralIntro} to $\omega$ of positive $d$-dimensional Lebesgue measure by assuming that $g$ and $V$ are analytic. In the very recent work \cite{BM23}, Burq and Moyano withdraw the analyticity on the metric $g$, replacing it by the sharp Lipschitz assumption,  but assuming $V=0$, and obtain \eqref{eq:PiLambdaSpectralIntro} for $\omega$ of positive $(d-\delta)$-Hausdorff content thanks to the new propagation of smallness results from \cite{LM18}.

The second goal of the paper is then to follow Burq, Moyano's strategy starting from our new propagation of smallness results for gradients of solutions of \eqref{eq:schrodingerequationIntro} to get new spectral estimates for the Schrödinger operator \eqref{eq:defHavM} in the compact setting.\medskip

On the Euclidean space, we are interesting in the following Schrödinger operator \begin{equation}
\label{eq:defHavRdIntro}
    H_{g, V, \kappa} u =- \frac{1}{\kappa(x)} \divergence(g^{-1}(x) \kappa(x) \nabla u) + V(x) u,\qquad x \in \R^d,
\end{equation}
where $g=g(x)$ is a symmetric uniformly elliptic matrix with Lipschitz entries, $\kappa=\kappa(x)$ is a positive bounded Lipschitz function and $V=V(x)$ is a bounded real-valued function. Notice that $H_{g,V, \kappa}$ is an unbounded self-adjoint operator on $L^2(\R^d, \kappa dx)$. As a consequence, one can still definite spectral projectors by 
\begin{equation}
    \label{eq:defspectralRdIntro}
    \Pi_{\Lambda} u = 1_{H_{g,V, \kappa}} u = \int_{-\|V\|_{\infty}}^{\Lambda} d m\qquad \forall u \in L^2(\R^d, \kappa dx),
\end{equation}
where $d m$ is the spectral measure of $H_{g,V, \kappa}$. Contrary to the case of compact manifolds, spectral inequalities of the form
\begin{equation}\label{eq:abstract_spectral_inequality}
\forall \Lambda >0,\ \exists C_{\Lambda}>0,\ \forall u \in L^2(\rr^d),\quad \|\Pi_{\Lambda} u\|_{L^2(\rr^d)} \leq C_{\Lambda} \|\Pi_{\Lambda}u\|_{L^2(\omega)},
\end{equation}
may require some geometric condition on $\omega$ to hold. When $g= I_d$, $\kappa=1$ and $V=0$, the Logvinenko-Sereda theorem \cite{LS74} shows that \eqref{eq:abstract_spectral_inequality} holds if and only if the measurable subset $\omega$ is thick. 
We say that $\omega$ is a thick subset of $\R^d$, if there exist $R, \gamma >0$ such that 
\begin{equation}
    \label{eq:measurethickIntro}
|\omega \cap B(x,R)| \geq \gamma |B(x,R)|, \qquad \forall x \in \R^d.
\end{equation}
Under this assumption and still in the case where $g=I_d$, $\kappa=1$ and $V=0$,  Kovrijkine establishes in \cite{Kov01} a quantitative version of the Logvinenko-Sereda theorem and shows the following inequality: for every $\Lambda >0$, 
\begin{equation}
\label{eq:measurableRdIntro}
   \norme{\Pi_{\Lambda} u}_{L^{2}(\R^d)} \leq \left(\frac{C_d}{\gamma}\right)^{C_d(1+R \sqrt{\Lambda})} \norme{ \Pi_{\Lambda} u }_{L^2(\omega)}\qquad \forall u \in L^2(\R^d),
\end{equation}
where $C_d>0$ is a positive constant depending only on the dimension.
Thanks to the estimate \eqref{eq:measurableRdIntro}, Egidi and Veselic \cite{EV18} and Wang, Wang, Zhang and Zhang \cite{WWZZ19} established that \eqref{eq:measurethickIntro} is actually a necessary and sufficient condition for the null-controllability of the associated parabolic equation. These results were generalized by Lebeau and Moyano in \cite{LM19} under analyticity assumption on $g$, $V$ and $\kappa$. Very recently, \cite{BM21} extended these results to the case of a Lipschitz metric $g$, a Lipschitz density $\kappa$ but without potential ($V=0$), again starting from propagation of smallness results from Logunov, Malinnikova. They were even able to deduce some spectral estimates under weaker assumptions on $\omega$ which allow it to be of Lebesgue measure zero.

The third goal of the paper is then to follow Burq, Moyano's strategy starting from our new propagation of smallness results for \eqref{eq:schrodingerequationIntro} to get new spectral estimates for the Schrödinger operator \eqref{eq:defHavRdIntro} in the non-compact setting.\medskip

\textbf{Acknowledgements}: Both authors are partially supported by the Project TRECOS ANR-20-CE40-0009 funded by the ANR (2021--2024).

\section{Main results}

The goal of this part is to state the main results of the paper that are quantitative propagation of smallness results for solutions to Schrödinger type equations in a bounded domain of $\R^d$ and their applications to spectral estimates for Schrödinger operators on compact Riemannian manifolds and on the whole euclidean space.\medskip

Recall that for $d \geq 0$, the $d$-Hausdorff content (or measure) of a set $E\subset \mathbb{R}^n$ is 
$$ \mathcal{C}_{\mathcal{H}}^d (E) = \inf \left\{ \sum_j r_j ^d\ ;\ E \subset \bigcup_j B(x_j, r_j)\right\},$$
and the Hausdorff dimension of $E$ is defined as 
$$ \text{dim}_\mathcal{H} (E) = \inf \{d \geq 0\ ;\    \mathcal{C}_{\mathcal{H}}^d(E) =0 \}.$$
We shall denote by $|E|$ the Lebesgue measure of the set $E$. Let us recall that the Hausdorff content of order $d$ is equivalent to the Lebesgue measure, 
$$ \exists C_d, c_d>0,\  \forall A \text{ borelian set},\  c_d |A| \leq \mathcal{C}_{\mathcal{H}} (A) \leq C_d |A|,$$
and  
\begin{equation}
\label{eq:equivalencehaussdorff}
  \mathcal{C}_{\mathcal{H}}^d (E)>0 \Rightarrow  \forall d'\in(0,d),\  \mathcal{C}_{\mathcal{H}}^{d'} (E)\geq \inf( 1,  \mathcal{C}_{\mathcal{H}}^d (E)).
  \end{equation}

\subsection{Propagation of smallness for the Schrödinger operator}

Let $\Omega$ be a bounded domain of $\R^d$. Let us consider the second order elliptic operator
\begin{equation}
\label{eq:defHavOmega}
    H_{A, V} u =- \divergence(A(x) \nabla u) + V(x) u,\qquad x \in \Omega,
\end{equation}
where $A=(a_{ij}(x))_{1, \leq i, j \leq d}$ is a symmetric uniformly elliptic matrix with Lipschitz entries
\begin{equation}
\label{eq:LipschitzAOmega}
    \Lambda_{1}^{-1} |\xi|^2 \leq \langle A(x) \xi, \xi \rangle \leq  \Lambda_{1} |\xi|^2,\quad |a_{ij}(x) - a_{ij}(y)| \leq \Lambda_2 |x-y|,\qquad x, y \in \Omega,\ \xi \in \R^d,
\end{equation}
for some $\Lambda_1, \Lambda_2>0$, where $|x|$ denotes the Euclidean norm of $x \in \rr^d$, and $V = V(x)$ is a real-valued bounded function, i.e. 
\begin{equation}
    \label{eq:hypothesisVOmega}
    V \in L^{\infty}(\Omega;\R).
\end{equation}

The first main result is about propagation of smallness for solutions to Schrödinger type equations.
\begin{theorem}
\label{cor:propagationsmallnessschrodinger}
Let $\rho, m, \delta >0$ and $\mathcal K, E \subset \Omega$, be measurable subsets such that 
\begin{equation}
    \label{eq:distanceandhaussdorffV}
    \text{dist}( \mathcal K, \partial \Omega) \geq \rho,\quad \text{dist}(E, \partial \Omega) \geq \rho,\quad \text{and} \quad  \mathcal C^{d-1+\delta}_{\mathcal H}(E) \geq m.
\end{equation}
There exist $C=C(\Omega,  \Lambda_1, \Lambda_2, \|V\|_{\infty}, \rho, m, \delta)>0$ and $ \alpha=\alpha(\Omega,  \Lambda_1, \Lambda_2, \|V\|_{\infty}, \rho, m, \delta) \in (0,1)$ such that for every weak solution $u \in W^{1,2}(\Omega) \cap L^{\infty}(\Omega)$ of the elliptic equation
\begin{equation}
\label{eq:elliptic_divergence_formV}
    - \divergence (A(x) \nabla u) + V(x) u=0 \ \text{in} \ \Omega,
\end{equation}
we have
\begin{equation}
\label{eq:PropagationSmallness_Divform_SolutionsV}
    \sup_{\mathcal K} | u| \leq C (\sup_{E} | u|)^{\alpha} (\sup_{\Omega} | u|)^{1-\alpha}.
\end{equation}
\end{theorem}

Let $\kappa=\kappa(x) \in W^{1, \infty}(\Omega)$ satisfying
\begin{equation}
\label{eq:LipschitzkappaOmega}
    \Lambda_{1}^{-1} \leq \kappa(x) \leq  \Lambda_{1} \quad \text{and} \quad |\kappa(x) - \kappa(y)| \leq \Lambda_2|x-y|, \qquad \forall x, y \in \rr^d.
\end{equation}

The second main result is about propagation of smallness of gradient of particular solutions to Schrödinger type equations.
\begin{theorem}
\label{cor:PropagationSmallness_withPotential}
There exists $\delta_d \in (0,1)$ depending only on the dimension $d$ such that the following holds. Let $\rho, m >0$, $\delta \in [0,\delta_d]$ and $\mathcal K, E \subset \Omega$ be measurable subsets such that 
\begin{equation}
    \label{eq:distanceandhaussdorffgradientV}
    \text{dist}( \mathcal K, \partial \Omega) \geq \rho,\quad \text{dist}(E, \partial \Omega) \geq \rho \quad \text{and} \quad  \mathcal C^{d-\delta}_{\mathcal H}(E) \geq m.
\end{equation}
There exist $C=C(\Omega,  \Lambda_1, \Lambda_2, \|V\|_{\infty}, \rho, m, \delta)>0$ and $ \alpha=\alpha(\Omega,  \Lambda_1, \Lambda_2, \|V\|_{\infty}, \rho, m, \delta) \in (0,1)$ such that for every weak solution $\hat{u}(x,t) \in W^{1,2}(\Omega \times (-1,1)) \cap L^{\infty}(\Omega \times (-1,1))$ of the elliptic equation
\begin{equation}
	\label{eq:EllipticEquation_NonDivForm}
		\left\{
			\begin{array}{ll}
				  - \divergence_x \cdot (A(x) \nabla_x \hat{u})- \kappa(x) \partial_{tt} \hat{u}+ V(x) \hat{u}=0 & \text{ in }  \Omega \times (-1,1), \\
				\hat{u}(x,0) = 0 & \text{ in } \Omega,
			\end{array}
		\right.
\end{equation}
we have
\begin{equation}
    \label{eq:PropagationSmallness_withPotential}
 \sup_{x \in \mathcal K} |\partial_t \hat{u}(x,0)|
     \leq C \Big(\sup_{x \in E} |\partial_t \hat{u}(x,0)|\Big)^\alpha \left\|\hat u\right\|^{1-\alpha}_{W^{1, \infty}(\Omega\times(-1,1))}.
\end{equation}
\end{theorem}

The new difficulty for proving \Cref{cor:propagationsmallnessschrodinger} and \Cref{cor:PropagationSmallness_withPotential} is that the results of \cite{LM18} are actually proved for divergence elliptic operators and their extensions to operators as in \eqref{eq:defHav} are not straightforward. One way to do it could consist in trying to adapt all the steps of their proof to a more general elliptic operator as considered here. But it is worth mentioning that \cite{LM18} is not self-contained as recalled by the authors and it uses some new deep results from \cite{Log18a} and \cite{Log18b}. We should also mention that one cannot expect to derive propagation of smallness for gradients of solutions to \eqref{eq:elliptic_divergence_formV} in full generality as in \cite[Theorem 5.1]{LM18} because as noted in \cite[Remark p. 362]{HHOHON99}, every closed subset of $\R^d$ could be the critical set of such a function so there is no hope to propagate smallness for gradients from sets of $(d-1-\delta)$-Hausdorff contents, even for small $\delta >0$. These are the main reasons why we decide to follow another easiest path that uses \cite[Theorems 2.1, 5.1]{LM18} as a black box.\medskip

We now present the main steps for the obtaining of \Cref{cor:propagationsmallnessschrodinger} and \Cref{cor:PropagationSmallness_withPotential}. Without loss of generality, we first reduce to the case $V \geq 0$. This reduction enables us to construct a positive multiplier that converts the Schrödinger type equation into a divergence elliptic equation.\medskip

\textbf{Reduction to the case $V \geq 0$.} We note that one can reduce the proof of \Cref{cor:propagationsmallnessschrodinger} and \Cref{cor:PropagationSmallness_withPotential} to the case 
\begin{equation}
\label{eq:Vsign}
V \geq 0.
\end{equation}
For \Cref{cor:propagationsmallnessschrodinger}, by looking at the function $\hat{u}(x,t)=u(x) \exp(\lambda t)$ that solves
\begin{equation}
\label{eq:elliptic_divergence_formVlambda}
    - \divergence_{x} \cdot (A(x) \nabla_x \hat{u}) - \partial_{tt} \hat{u}  + (V(x)+ \lambda^2) \hat{u}=0 \ \text{in} \ \Omega \times (-1,+1).
\end{equation}
For $\lambda \geq \|V\|_{\infty}^{1/2}$, we have that $\hat{V}(t,x) = V(x) + \lambda^2 \geq 0$, then one can apply Theorem \ref{cor:propagationsmallnessschrodinger} with $\hat{\Omega} = \Omega \times (-1,+1)$, $\hat{K} = K \times (-1/2,+1/2) \subset \subset \hat{\Omega}$, $\hat{E} = E \times (-1/2,+1/2) \subset \subset \hat{\Omega}$ to $\hat{u}$ satisfying \eqref{eq:elliptic_divergence_formVlambda}. For \Cref{cor:PropagationSmallness_withPotential}, the argument is in the same spirit by adding a ghost variable considering $\hat{v}(x,y,t) = \hat{u}(x,t) \exp(\lambda y)$ for some $\lambda \geq \|V\|_{\infty}^{1/2}$ and applying \Cref{cor:PropagationSmallness_withPotential} with $\hat{\Omega} =  \Omega \times (-1,1)$, $\hat{K} = K \times (-1/2,1/2)
\subset \subset \hat{\Omega}$ and $\hat{E} = E \times (-1/2,1/2) \subset \subset \hat{\Omega}$ to $\hat{v}$. Therefore in all the following, we will only consider the case \eqref{eq:Vsign}.\medskip

\textbf{Reduction to a divergence elliptic equation.} The key ingredient in the proofs of Theorem~\ref{cor:propagationsmallnessschrodinger} and \ref{cor:PropagationSmallness_withPotential} consists in constructing a suitable positive multiplier to the equation $H_{A,V} \phi = 0$, in the case when $V \geq 0$, that is showing the existence of $\phi \in W^{1, \infty}(\Omega_0)$ satisfying
\begin{equation*}
    - \divergence(A \nabla \phi) + V \phi = 0 \ \text{in}\ \Omega_0 \quad \text{and} \quad \phi >0\ \text{in}\ \Omega_0,
\end{equation*}
where $\Omega_0$ is a smooth domain satisfying
\begin{equation*}
 \text{dist}( \mathcal K, \partial \Omega_0) \geq \rho/2,\  \text{dist}(E, \partial \Omega_0) \geq \rho/2,\ \text{and}\ \Omega_0 \subset \subset \Omega.
\end{equation*}
This enables us to reduce the obtaining of propagation of smallness of solutions to \eqref{eq:elliptic_divergence_formV} to the application of propagation of smallness of solutions to divergence elliptic equations for $v=u/\phi$. Indeed, $v$ now satisfies
\begin{equation*}
    - \divergence(\phi^2 A \nabla v) = 0\ \text{in}\ \Omega_0.
\end{equation*}
Thanks to suitable lower and upper bounds on $\phi$, these propagation of smallness estimates obtained on $v$ provides estimates on $u$. For the case when $u$ satisfies \eqref{eq:EllipticEquation_NonDivForm}, the same strategy works but now the equation satisfied by $\hat{v} = \hat{u} / \phi$ is
\begin{equation*}
				  - \divergence_x \cdot (\phi^2 A(x) \nabla_x \hat{v})-  \partial_{t} ( \phi^2 \kappa(x) \partial_t \hat{v}) =0 \ \text{ in }  \Omega_0 \times (-1,1),
\end{equation*}
that is a divergence elliptic equation with the extra condition $\hat{v}(x,0)=0$. Propagation of smallness estimates on $|\nabla_{t,x} v|$ on particular sets will then provide the expected result \eqref{eq:PropagationSmallness_withPotential}.

\medskip

The following remarks are in order.

The main advantage of our proof is that we obtain propagation of smallness for solutions to Schrödinger type equations in a general setting without redoing all the arguments of Logunov and Malinnikova. But the main drawback of such a strategy is that for gradient of solutions, the application of propagation of smallness for gradients \cite[Theorem 5.1]{LM18} is applied to $v=u/\phi$ from which it is difficult to deduce estimates on $|\nabla u|$. Actually, this is not only a technical difficulty since, as recalled previously, one cannot expect to obtain propagation of smallness for gradients in full generality. Nevertheless, this method allows us to deal with the particular setting of \eqref{eq:EllipticEquation_NonDivForm} and to obtain \Cref{cor:PropagationSmallness_withPotential}. Fortunately, propagation of smallness estimates \eqref{eq:PropagationSmallness_withPotential} are sufficient for our applications to spectral estimates. 

Finally, we would like to highlight the very recent preprint \cite{Zhu24} from which one can also obtain Theorems~\ref{cor:propagationsmallnessschrodinger} and \ref{cor:PropagationSmallness_withPotential}, assuming that $V \in W^{1, \infty}(\Omega;\R)$, starting from \cite[Theorem 2.1 and Theorem 5.1]{LM18} as a black box. His strategy is rather different from ours because it consists in putting the zero order term $V$ in the principal part of the operator by adding a ghost variable. The Lipschitz assumption on $V$ seems to be difficult to remove with such a method.

Last but not least, the parameter $\delta_d \in (0,1)$ appearing in \Cref{cor:PropagationSmallness_withPotential} is small a priori and actually comes from \cite[Theorem 5.1]{LM18}. The extension to an arbitrary $\delta_d \in (0,1)$ is an open and very likely difficult open problem. However, it is worth mentioning that for $d=1$, we can take an arbitrary $\delta \in (0,1)$ by using the propagation of smallness result for gradients \cite[Theorem 1.2]{Zhu23} in dimension $d+1=2$ instead of \cite[Theorem 5.1]{LM18}, together with our strategy of reduction to divergence elliptic equation. The same remark will apply in the next for spectral estimates and applications. See also \cite{SSY23} for a similar strategy when $d=1$ but with the use of \cite[Theorem 1.1]{Zhu23}.

\subsection{Spectral estimates on compact manifolds and applications}\label{sec:Main_result_spectral_estimates_on_manifold}

Let $M$ be a $C^1 \cap W^{2,\infty}$, connected, compact manifold of dimension $d \geq 1$, possibly with boundary, equipped with a Riemannian metric $g$. In this section, we fix an Atlas $\mathcal A= (\mathcal V_\sigma, \Psi_\sigma)_{\sigma \in \mathcal J}$ containing a finite number of charts with $(W^{2, \infty}\cap  C^1)$-diffeomorphisms $\Psi_\sigma :  \mathcal V_\sigma \longrightarrow \Psi_\sigma(\mathcal V_{\sigma}) \subset \rr^{d-1} \times \rr_+$ such that there exists a family of open sets $(\mathcal U_\sigma)_{\sigma \in \mathcal J}$ satisfying 
\begin{equation}\label{eq:charts_recovering0}
M = \bigcup_{\sigma \in \mathcal J} \mathcal U_\sigma,
\end{equation}
and such that $ \mathcal U_\sigma$ is compactly included in the open set $  \mathcal V_\sigma$ in $M$, for all $\sigma \in \mathcal J$. In the case when $M$ is assumed to be without boundary, then for any $\sigma  \in \mathcal J$, $\Psi(\mathcal V_{\sigma})$ is an open set of $\rr^d$.

Let us consider the second order elliptic operator
\begin{equation}
\label{eq:defHav}
   H_{g, V} u =-\Delta_g u + V(x) u,\qquad x \in M,
\end{equation}
where $V = V(x)$ is a real-valued bounded function, i.e. 
\begin{equation}
    \label{eq:hypothesisVM}
    V \in L^{\infty}(M;\R),
\end{equation} and $g$ is assumed to be $\Lambda_1$-elliptic and $\Lambda_2$-Lipschitz, in the sense that if $(g^{\sigma}_{i,j})_{1 \leq i,j \leq d}$ are the local coordinates of $g$ in a local chart $(\mathcal V_\sigma, \Psi_\sigma)$, 

\begin{equation}\label{eq:EllipticM}
    \Lambda_{1}^{-1} |\xi|^2 \leq \sum_{i,j} g^{\sigma}_{i,j}(\Psi_\sigma^{-1}(x)) \xi_i\xi_j \leq  \Lambda_{1} |\xi|^2, \qquad \forall x \in \Psi_\sigma(\mathcal V_\sigma), \xi \in \rr^d,
\end{equation}
and
\begin{equation}
\label{eq:LipschitzM}
    |g^{\sigma}_{ij} \circ \Psi_\sigma^{-1}(x) - g^{\sigma}_{ij}\circ \Psi_\sigma^{-1}(y)| \leq \Lambda_2 |x-y|, \qquad \forall x, y \in \Psi_{\sigma}(\mathcal V_\sigma),
\end{equation}
for some $\Lambda_1>0$ and $\Lambda_2 >0$. 

Let us define 
\begin{equation}
    \text{Dom}(H_{g, V}) = \{ u \in H^2(M)\ ;\ u = 0\ \text{on}\ \partial M\ \text{or}\ \partial_{\nu} u = 0\ \text{on}\ \partial M\}.
\end{equation}
Note that if $\partial M = \emptyset$ then $\mathrm{Dom}(H_{g, V}) = H^2(M)$. Under these assumptions, this is well-known that $H_{g, V}$ admits an orthonormal basis in $L^2(M)$ of eigenfunctions denoted by $(\phi_k)_{k \geq 1}$, associated to the sequence of real eigenvalues $(\lambda_k)_{k \geq 1}$. Given $\Lambda >0$, we introduce the spectral projector $\Pi_{\Lambda}$ as follows
\begin{equation}
    \Pi_{\Lambda} u = \sum_{\lambda_k \leq \Lambda} \langle u, \phi_k\rangle_{L^2} \phi_k,\qquad \forall u   \in L^2(M).
\end{equation}

Our first main result states the following spectral estimates for the Schrödinger operator \eqref{eq:defHav}.

\begin{theorem}
\label{thm:spectralcompact}

There exists $\delta_d \in (0,1)$ such that for all $\delta \in [0, \delta_d]$, for every observation set $\omega \subset M$ satisfying $\mathcal{C}_{\mathcal H}^{d-\delta}(\omega) \geq m > 0$, there exists $C=C(M,g,V,\mathcal A,\delta,m)>0$ such that for every $\Lambda >0$, we have
\begin{equation}
\label{eq:hausdorffcompact}
   \norme{\Pi_{\Lambda} u}_{L^{\infty}(M)} \leq C e^{C \sqrt{\Lambda}} \sup_{x \in \omega} \left|(\Pi_{\Lambda} u)(x) \right|\qquad \forall u \in L^2(M).
\end{equation}

In particular, for every measurable set $\omega \subset M$ satisfying $|\omega| \geq m>0$, there exists $C=C(M,g,V,\mathcal A , m)>0$ such that for every $\Lambda >0$, we have 
\begin{equation}
\label{eq:measurablecompact}
   \norme{\Pi_{\Lambda} u}_{L^{\infty}(M)} \leq C e^{C \sqrt{\Lambda}} \norme{ \Pi_{\Lambda} u }_{L^1(\omega)}\qquad \forall u \in L^2(M).
\end{equation}

\end{theorem}

The following comments are in order. First, Theorem \ref{thm:spectralcompact} generalizes \cite[Theorem 1]{BM23} to the case of the Schrödinger operator as in \eqref{eq:defHav}. Secondly, the inequality \eqref{eq:measurablecompact} is a $L^{\infty}$-$L^1$ spectral estimate from which one can easily deduce the more standard $L^{2}$-$L^2$ spectral estimate as recalled in \eqref{eq:PiLambdaSpectralIntro} by using the continuous embeddings $L^{\infty}(M) \hookrightarrow L^2(M)$ and $L^2(\omega) \hookrightarrow L^1(\omega)$
\begin{equation}
\label{eq:measurablecompactL2}
   \norme{\Pi_{\Lambda} u}_{L^{2}(M)} \leq C e^{C \sqrt{\Lambda}} \norme{\Pi_{\Lambda} u}_{L^2(\omega)}\qquad \forall u \in L^2(M).
\end{equation}
 On the other hand, the inequality \eqref{eq:hausdorffcompact} is a $L^{\infty}$-$L^{\infty}$ spectral estimate from which one can only deduce a $L^{2}$-$L^{\infty}$ spectral estimate
\begin{equation}
\label{eq:hausdorffcompactL2}
   \norme{\Pi_{\Lambda} u}_{L^{2}(M)} \leq  C e^{C \sqrt{\Lambda}} \sup_{x \in \omega} \left|(\Pi_{\Lambda} u)(x) \right|\qquad \forall u \in L^2(M).
\end{equation}
Moreover, without extra assumption on $\omega$ here, there is no hope to transform \eqref{eq:hausdorffcompactL2} into a $L^{2}$-$L^{2}$ estimate because one can have $|\omega|=0$. Last but not least, the parameter $\delta \in (0,1)$ is small a priori and actually comes from \Cref{cor:PropagationSmallness_withPotential} so from \cite[Theorem 5.1]{LM18}. The extension to an arbitrary $\delta \in (0,1)$ is an open and very likely difficult open problem.

The strategy of the proof of Theorem \ref{thm:spectralcompact} will follow the one of \cite[Theorem 1]{BM23} that uses propagation of smallness for gradient of solution to elliptic equations from \cite[Theorem 5.1]{LM18}. The new difficulty here is that the results of \cite{LM18} are actually proved for divergence elliptic operators and their extensions to operators as in \eqref{eq:defHav} are not straightforward. This is why we will actually use our new Theorem \ref{cor:PropagationSmallness_withPotential}.\medskip

We now focus on the time evolution equation
    \begin{equation}
	\label{eq:heat_manifold}
		\left\{
			\begin{array}{ll}
				  \partial_t u  + H_{g,V} u=0 & \text{ in }  (0,+\infty) \times M, \\
				u(0, \cdot) = u_0 & \text{ in }M,
			\end{array}
		\right.
\end{equation}
completed with homogeneous Dirichlet or Neumann boundary conditions if $\partial M  \neq \emptyset$. 

Our second main result is the establishment of the following observability inequalities.
\begin{theorem}

There exists $\delta_d \in (0,1)$ such that for all $\delta \in [0, \delta_d]$, $m>0$ and for every measurable set $\omega \subset M$ satisfying $\mathcal{C}_{\mathcal H}^{d-\delta}(\omega) \geq m >0$, there exists a positive constant $C=C(M,g,V,\mathcal A,\delta,m)>0$ such that for every $T \in (0,1)$, $u_0 \in L^2(M)$, the solution $u \in C([0,T];L^2(M))$ of \eqref{eq:heat_manifold} satisfies 
\begin{equation}
\label{eq:obsheatmanifoldhausdorff}
    \|u(T, \cdot)\|_{L^2(M)}^2 \leq C e^{\frac CT} \int_0^T \left(\sup_{x \in \omega} |u(t,x)|\right)^2 dx dt.
\end{equation}

For every measurable set $\omega \subset M$ satisfying $|\omega| \geq m>0$, there exists a positive constant $C =C(M,g,V,\mathcal A,m)>0$ such that for every $T \in (0,1)$, $u_0 \in L^2(M)$, the solution $u \in C([0,T];L^2(M))$ of \eqref{eq:heat_manifold} satisfies
\begin{equation}
\label{eq:obsheatmanifold}
    \|u(T, \cdot)\|^2_{L^2(M)} \leq C e^{\frac CT} \int_0^T \|u(t, \cdot)\|^2_{L^2(\omega)} dt.
\end{equation}
\end{theorem}
By using the spectral estimates \eqref{eq:measurablecompactL2}, the proof of \eqref{eq:obsheatmanifold} is by now classical and originally comes from the Lebeau-Robbiano method for obtaining the null-controllability of the heat equation starting from a spectral estimate, see \cite{LR95} and \cite[Section 6]{LRL12}. This strategy was latter extended by Miller in \cite{Mil10}. Therefore, \eqref{eq:obsheatmanifold} directly comes from \cite[Theorem 2.2]{Mil10} or \cite[Theorem 2.1]{BPS18}. The proof of \eqref{eq:obsheatmanifoldhausdorff} is in the same spirit but adapted to the the particular functional setting of the $L^2-L^{\infty}$ spectral estimate \eqref{eq:hausdorffcompactL2}, see \cite[Section 4]{BM23} for details. The restriction $T \in (0,1)$ that we will keep in the following is simply for the obtaining of the constant of observability in small time of the form $C e^{\frac CT}$. \medskip

For a measurable $\omega \subset M$, we finally focus on the controlled system
    \begin{equation}
	\label{eq:heat_manifold_controlled}
		\left\{
			\begin{array}{ll}
				  \partial_t y  + H_{g,V} y= h 1_{\omega} & \text{ in }  (0,+\infty) \times M, \\
				y(0, \cdot) = y_0 & \text{ in }M,
			\end{array}
		\right.
\end{equation}
completed with homogeneous Dirichlet or Neumann boundary conditions if $\partial M  \neq \emptyset$. In \eqref{eq:heat_manifold_controlled}, at time $t \in [0,+\infty)$, $y(t,\cdot) : M \to \R$ is the state and $h(t,\cdot) : \omega \to \R$ is the control. In the following, we denote by $\mathcal M(M)$ the space of Borel measure on $M$.

Our last main result of this section provides null-controllability results for \eqref{eq:heat_manifold_controlled}.
\begin{theorem}
\label{thm:controlledmanifold}

There exists $\delta_d \in (0,1)$ such that for all $\delta \in [0, \delta_d]$, $m>0$ and for every measurable set $\omega \subset M$ satisfying $\mathcal{C}_{\mathcal H}^{d-\delta}(\omega) \geq m >0 $, there exists $C=C(M,g,V,\mathcal A, \delta,m)>0$ such that for every $T \in (0,1)$ and $y_0 \in L^2(M)$, there exists $h \in L^2(0,T;\mathcal M(M))$ supported in $(0,T)\times \omega$ satisfying
\begin{equation}
    \label{eq:estimatecontrolmanifoldswild}
    \int_0^T \|h(t)\|_{\mathcal M(M)}^2 dt \leq C  e^{\frac CT} \norme{y_0}_{L^2(M)}^2,
\end{equation}
such that the solution $y$ of \eqref{eq:heat_manifold_controlled} satisfies $y(T,\cdot) = 0$.

For every measurable set $\omega \subset M$ satisfying $|\omega| \geq m >0$, there exists a positive constant $C = C(M,g,V,\mathcal A, m)>0$ such that for every $T \in (0,1)$, $y_0 \in L^2(M)$, there $h \in L^2(0,T;L^2(\omega))$ satisfying
\begin{equation}
    \label{eq:estimatecontrolmanifolds}
    \norme{h}_{L^2(0,T;L^2(\omega))} \leq C  e^{\frac CT} \norme{y_0}_{L^2(M)},
\end{equation}
such that the solution $y \in C([0,T];L^2(M))$ of \eqref{eq:heat_manifold_controlled} satisfies $y(T,\cdot) = 0$.
\end{theorem}
Recall that the norm of the space of Borel measures on the metric space $\omega$, denoted by $\mathcal{M}(\omega)$, is defined as
\begin{equation}
    \label{eq:normeborel}
    \| \mu \|_{\mathcal M(M)} = \sup_{f \in C^0(M)} \frac{\left|\int_{M} f d \mu \right|}{\|f\|_{\infty}}.
\end{equation}

The second part of Theorem \ref{thm:controlledmanifold} comes from a classical duality argument together with the use of the observability estimate \eqref{eq:obsheatmanifold}, see for instance \cite[Theorem 2.44]{Cor07}. The first part is less standard due to the functional setting but details can be found in \cite[Section 5]{BM23}. 

\subsection{Spectral estimates on the Euclidean space and applications}

Let us consider the second order elliptic operator
\begin{equation}
\label{eq:defHavRd}
    H_{g, V, \kappa} u =- \frac{1}{\kappa(x)} \divergence(\kappa(x) g^{-1}(x)  \nabla u) + V(x) u,\qquad x \in \R^d,
\end{equation}
where $g(x)=(g_{ij}(x))_{1, \leq i, j \leq d}$ is a symmetric uniformly elliptic matrix with Lipschitz entries
\begin{equation}
\label{eq:LipschitzRd}
    \Lambda_{1}^{-1} |\xi|^2 \leq \langle g(x) \xi, \xi \rangle \leq  \Lambda_{1} |\xi|^2,\quad |g_{ij}(x) - g_{ij}(y)| \leq \Lambda_2 |x-y|, \qquad x, y \in \R^d,\ \xi \in \R^d,
\end{equation}
for some $\Lambda_1, \Lambda_2 >0$, $V = V(x)$ is a real-valued bounded function, i.e.
\begin{equation}
\label{eq:boundVRd}
  V  \in L^{\infty}(\rr^d; \rr),
\end{equation}
and $\kappa \in W^{1,\infty}(\rr^d, \rr^*_+)$ is a positive bounded Lipschitz density satisfying
\begin{equation}
    \label{eq:hypothesiskappa}
\Lambda_{1}^{-1} \leq \kappa \leq  \Lambda_{1}.
\end{equation}  
The operator $H_{g,V, \kappa}$ is an unbounded self-adjoint operator on $L^2(\R^d, \kappa dx)$ with domain $H^2(\R^d)$. Notice that, under the assumption \eqref{eq:hypothesiskappa}, we have $$ \Lambda_{1}^{-\frac 12} \|\cdot \|_{L^2(\rr^d)} \leq \| \cdot \|_{L^2(\rr^d, \kappa dx)} \leq  \Lambda_{1}^{\frac 12} \| \cdot \|_{L^2(\rr^d)},$$ and $L^2(\rr^d)= L^2(\rr^d, \kappa dx)$. However, let us insist on the fact that the self-adjointness of $H_{g, V, \kappa}$ is related to the scalar product
$$\langle f, g \rangle_{L^2(\rr^d, \kappa dx)} = \int_{\rr^d} f(x) g(x) \kappa(x) dx.$$
Given $\Lambda >0$, we introduce the spectral projector as follows
\begin{equation}
    \label{eq:defspectralRd}
    \Pi_{\Lambda} u = 1_{H_{g,V, \kappa}} u = \int_{-\|V\|_{\infty}}^{\Lambda} d m\qquad \forall u \in L^2(\R^d, \kappa dx),
\end{equation}
where $d m$ is the spectral measure of $H_{g,V, \kappa}$.

In the sequel, we need the following definition.
\begin{definition}
Let $R>0$ and $0<\gamma\leq 1$. A measurable subset $\omega \subset \R^d$ is said to be $\gamma$-thick at scale $R$ if
\begin{equation}
    \label{eq:measurethick}
|\omega \cap B(x,R)| \geq \gamma |B(x,R)|,\qquad \forall x \in \R^d.
\end{equation}
A subset $\omega \subset \rr^d$ is said to be thick if and only if it is $\gamma$-thick at scale $R$ for some $R>0$ and $0<\gamma \leq 1$.

Let $R>0$, $m>0$, and $0< d' \leq d$. A measurable set $\omega \subset \R^d$ is said to be $(m, d')$-uniformly distributed at scale $R$ if
\begin{equation}
    \label{eq:haussdorffmeasureRd}
\mathcal{C}_{\mathcal{H}}^{d'}(\omega \cap B(x,R)) \geq m, \qquad \forall x \in \R^d.
\end{equation}
A subset $\omega \subset \rr^d$ is said to be $d'$-uniformly distributed if and only if it is $(m, d')$-uniformly distributed at scale $R$ for some $R,m >0$.
\end{definition}
Notice that, since the $d$-dimensional Hausdorff measure is equivalent to the Lebesgue measure of $\rr^d$, a subset $\omega \subset \rr^d$ is $d$-uniformly distributed if and only if it is thick. In \eqref{eq:measurethick} and \eqref{eq:haussdorffmeasureRd}, the parameter $R>0$ will be always chosen such that
\begin{equation}
    \label{eq:recouvrement}
    \R^d = \bigcup_{k \in \Z^d} B(k,R).
\end{equation}

Our main result of this section states the following spectral estimates for the Schrödinger operator \eqref{eq:defHavRd} with particular observation sets. 

\begin{theorem}
\label{thm:spectralRd}
Let $R>0$ be such that \eqref{eq:recouvrement} holds, $0<\gamma \leq 1$ and $m>0$.

There exists $\delta_d \in (0,1)$ such that for all $\delta \in [0, \delta_d]$ and for every $(m, d-\delta)$-uniformly distributed set $\omega \subset \rr^d$ at scale $R$, there exists $C=C(\Lambda_1, \Lambda_2, \|V\|_{L^{\infty}},R,m, \delta)>0$ such that for every $\Lambda >0$, we have
\begin{equation}
\label{eq:hausdorffRd}
   \norme{\Pi_{\Lambda} u}_{L^{2}(\R^d)} \leq C e^{C \sqrt{\Lambda}} \sum_{k \in \Z^d} \sup_{x \in \omega \cap B( k,R)} \left|(\Pi_{\Lambda} u)(x) \right|\qquad \forall u \in L^2(\R^d).
\end{equation}

For every $\gamma$-thick set $\omega \subset \R^d$ at scale $R$, there exists $C=C(\Lambda_1, \Lambda_2, \|V\|_{L^{\infty}},R,\gamma)>0$ such that for every $\Lambda >0$, we have 
\begin{equation}
\label{eq:measurableRd}
   \norme{\Pi_{\Lambda} u}_{L^{2}(\R^d)} \leq C e^{C \sqrt{\Lambda}} \norme{ \Pi_{\Lambda} u }_{L^2(\omega)}\qquad \forall u \in L^2(\R^d).
\end{equation}

\end{theorem}

The following remarks are in order. First, Theorem \ref{thm:spectralRd} generalizes \cite[Theorem 1]{BM21} to the case of the Schrödinger operator as in \eqref{eq:defHavRd}. Note that \eqref{eq:measurableRd} has been previously established in \cite{LM19} assuming some analyticity condition on the potential $V$ and in \cite{SSY23} in the one-dimensional case by exploiting the recent improvement of propagation of smallness for solutions to elliptic equations in dimension $d=2$ in \cite{Zhu23}. See also \cite{AS23}, \cite{Wan24} and \cite{Zhu24} for the case of unbounded potentials $V$. Note in particular that \cite{Zhu24} obtains \eqref{eq:measurableRd} for Lipschitz bounded potentials $V$ by an adaptation of the propagation of smallness argument for gradients of \cite[Theorem 5.1]{LM18} but his approach is different from us because it consists in putting the zero-order perturbation term $V$ in the principal part of the elliptic operator, that explains the Lipschitz assumption on $V$. The same remark as in the compact setting can be done for the parameter $\delta \in (0,1)$ appearing for the proof of \eqref{eq:hausdorffRd}.

The strategy of the proof of Theorem \ref{thm:spectralRd} will follow the one of \cite[Theorem 1]{BM21} and again the new difficulty comes from the fact that the results of \cite{LM18} are actually proved for divergence elliptic operators. \medskip

We now focus on the evolution equation
    \begin{equation}
	\label{eq:heat_Rd}
		\left\{
			\begin{array}{ll}
				  \partial_t u  + H_{g,V, \kappa} u=0 & \text{ in }  (0,+\infty) \times \R^d, \\
				u(0, \cdot) = u_0 & \text{ in }\R^d.
			\end{array}
		\right.
\end{equation}

Our second main result is the establishment of the following observability inequalities.
\begin{theorem}
\label{thm:obsRd}
Let $R>0$ be such that \eqref{eq:recouvrement} holds, $0<\gamma \leq 1$ and $m>0$.

There exists $\delta_d \in (0,1)$ such that for all $\delta \in [0, \delta_d]$ and for every $(m, d-\delta)$-uniformly distributed set $\omega \subset \rr^d$ at scale $R$, there exists $C=C(\Lambda_1, \Lambda_2, \|V\|_{L^{\infty}},R,m, \delta)>0$ such that for every $T \in (0,1)$ and $u_0 \in L^2(\rr^d)$, the mild solution $u \in C([0,T];L^2(\R^d))$ of \eqref{eq:heat_Rd} satisfies 
\begin{equation}
\label{eq:obshaussdorff} 
\| u(T, \cdot) \|^2_{L^2(\R^d)} \leq  C e^{\frac CT} \sum_{k \in \mathbb{Z}^d} \int_0^T \sup_{ x\in \omega \cap B( k, R)} | u|^2(t,x)  dt.
\end{equation}

For every $\gamma$-thick set $\omega \subset \R^d$ at scale $R$, there exists $C=C(\Lambda_1, \Lambda_2, \|V\|_{L^{\infty}},R,\gamma)>0$ such that for every $T \in (0,1)$ and $u_0 \in L^2(\R^d)$, the solution $u \in C([0,T];L^2(\R^d))$ of \eqref{eq:heat_Rd} satisfies
\begin{equation}
\label{eq:obsheatRd}
    \|u(T, \cdot)\|^2_{L^2(\R^d)} \leq C e^{\frac CT} \int_0^T \|u(t, \cdot)\|^2_{L^2(\omega)} dt.
\end{equation}

\end{theorem}
 This result is in particular a generalization of \cite{DYZ21} and \cite{BM21} by the adding of the zero-order term $V$. This constitutes also an improvement of \cite{NTTV20} and \cite{DWZ20} which consider the case where $\omega$ is a union of disjoint open balls. Note that the thickness condition turns out to be necessary for \eqref{eq:obsheatRd} for $g=I_d$ and $V=0$  by recalling \cite{EV18} and \cite{WWZZ19}. Other necessary and sufficient conditions were derived in \cite{BEPS20} when looking at controls of the form $h 1_{\omega(t)}$. For a complete proof of \Cref{thm:obsRd} from the spectral estimates of \Cref{thm:spectralRd}, see \cite{BM21} and the references therein.
\medskip

For a measurable $\omega \subset \R^d$, we finally focus on the controlled system
    \begin{equation}
	\label{eq:heat_Rd_controlled}
		\left\{
			\begin{array}{ll}
				  \partial_t y  + H_{g,V, \kappa} y= h 1_{\omega} & \text{ in }  (0,+\infty) \times \R^d, \\
				y(0, \cdot) = y_0 & \text{ in }\R^d,
			\end{array}
		\right.
\end{equation}
In \eqref{eq:heat_Rd_controlled}, at time $t \in [0,+\infty)$, $y(t,\cdot) : \R^d \to \R$ is the state and $h(t,\cdot) : \omega \to \R$ is the control.

Let us denote by $\mathcal M(\rr^d)$ the space of Borel measures on $\rr^d$. Our last main result provides null-controllability results for \eqref{eq:heat_Rd_controlled}.
\begin{theorem}
\label{thm:controlledRd}
Let $R>0$ be such that \eqref{eq:recouvrement} holds, $0<\gamma \leq 1$ and $m>0$.

There exists $\delta_d \in (0,1)$ such that for all $\delta \in [0, \delta_d]$ and for every $(m, d-\delta)$-uniformly distributed set $\omega \subset \rr^d$ at scale $R$, there exists $C=C(g,V, \kappa,R,m, \delta)>0$ such that for every $T \in (0,1)$ and $y_0 \in L^2(\R^d)$, there exists $h \in L^2(0,T;\mathcal M(\R^d))$ supported in $(0,T)\times \omega$ satisfying
\begin{equation}
    \label{eq:estimatecontrolmanifoldswildRd}
   \sum_{k \in \mathbb{Z}^d}  \int_0^T \|h(t)\|_{\mathcal M(B(k,R))}^2 dt \leq C  e^{\frac CT} \norme{y_0}_{L^2(\R^d)}^2,
\end{equation}
such that the solution $y \in C([0,T];L^2(\R^d))$ of \eqref{eq:heat_Rd_controlled} satisfies $y(T,\cdot) = 0$.

For every $\gamma$-thick set $\omega \subset \R^d$ at scale $R$, there exists $C=C(g,V,\kappa,R,\gamma)>0$ such that for every $T \in (0,1)$ and $y_0 \in L^2(\R^d)$, there $h \in L^2(0,T;L^2(\omega))$ satisfying
\begin{equation}
    \label{eq:estimatecontrolRd}
    \norme{h}_{L^2(0,T;L^2(\omega))} \leq C  e^{\frac CT} \norme{y_0}_{L^2(\R^d)},
\end{equation}
such that the solution $y \in C([0,T];L^2(\R^d))$ of \eqref{eq:heat_Rd_controlled} satisfies $y(T,\cdot) = 0$.

\end{theorem}
In \eqref{eq:estimatecontrolmanifoldswildRd}, $\|h(t)\|_{\mathcal M(B(k,R))}$ is the norm of the restriction of $h(t)$ on $B(k,R)$, which is therefore supported on $\omega \cap B(k,R)$, defined as 
$$\|h(t)\|_{\mathcal M(B(k,R))} = \sup_{f \in C^0(B(k,R))} \frac{\int_{B(k,R)} f dh(t)}{\|f\|_{L^{\infty}(B(k,R))}}.$$
For a complete proof of \Cref{thm:controlledRd} from the observability inequalities of \Cref{thm:obsRd}, see \cite{BM21} and the references therein.

\section{Proof of propagation of smallness for Schrödinger operators}

The goal of this section is to establish quantitative propagation of smallness for solutions to Schrödinger operators. 

Let $\Omega$ be a bounded domain of $\R^d$. We consider the elliptic operator $H_{A,V}$ defined in \eqref{eq:defHavOmega}, with the Lipschitz assumption on $A$, i.e. \eqref{eq:LipschitzAOmega} and the boundedness assumption on $V$, i.e. \eqref{eq:hypothesisVOmega}. Moreover, as explained under Theorem~\ref{cor:PropagationSmallness_withPotential}, one can assume that $V \geq 0$.\medskip

\textbf{Reduction to a smooth bounded domain $\Omega_0$.} Given $\mathcal{K}$, $E$ such that \eqref{eq:distanceandhaussdorffV} or \eqref{eq:distanceandhaussdorffgradientV} hold then from \cite[Proposition 8.2.1]{Dan08}, one can find a $C^{\infty}$-domain $\Omega_0$ such that 
\begin{equation}
\label{eq:Omega0}
 \text{dist}( \mathcal K, \partial \Omega_0) \geq \rho/2,\  \text{dist}(E, \partial \Omega_0) \geq \rho/2,\ \text{and}\ \Omega_0 \subset \subset \Omega.
\end{equation}

\subsection{Existence of a positive multiplier and reduction to divergence form}

In this part, we will construct the positive multiplier $\phi$ in $\Omega_0$ then reduce the Schrödinger type equation to a divergence elliptic equation in $\Omega_0$ with the help of the multiplier.

The following result is quite standard.

\begin{lemma}
\label{lem:auxiliary_fonction_for_elliptic_equation}
There exists $\phi \in W^{1, \infty}(\Omega_0)$ satisfying
\begin{equation}
	\label{eq:equationmultiplier}
				- \divergence \left(A(x) \nabla \phi \right) + V(x) \phi=0\  \text{in}\ \Omega_0.
\end{equation}
Moreover, there exists $c=c(\Omega,\Omega_0,\Lambda_1, \Lambda_2,\|V\|_{\infty})>0$ and $C =C(\Omega,\Omega_0,\Lambda_1, \Lambda_2,\|V\|_{\infty})>0 $ such that
\begin{equation}
\label{eq:lowerbound}
    \phi \geq c \ \text{in}\ \Omega_0, 
\end{equation}
and
\begin{equation}
\label{eq:Winftyboundphi}
    \|\phi\|_{W^{1,\infty}(\Omega_0)} \leq C.
\end{equation}
\end{lemma}
\begin{proof}
We first solve the boundary elliptic problem, by \cite[Theorem 8.3]{GT01}, there exists $\phi \in W^{1,2}(\Omega)$,
\begin{equation}
	\label{eq:equationmultiplierBoundary}
				- \divergence \left(A(x) \nabla \phi \right) + V(x) \phi=0\  \text{in}\ \Omega,\qquad \phi = 1\ \text{on}\ \partial\Omega.
\end{equation}

Let us take $x_0 \in \R^d$ such that for every $x = (x_1, \dots, x_d) \in \Omega$, we have $x_1 - x_0 \geq 0$. Let $\lambda \geq 0$ and let us then define
$$ \phi_{-}(x) = \exp(-\lambda(x_1-x_0)),\ \phi_{+}(x) = 1\qquad \forall x=(x_1, \dots, x_d) \in \Omega.$$
Then, $\phi_{-}$, respectively $\phi_{+}$, is a subsolution, respectively a supersolution, to \eqref{eq:equationmultiplierBoundary} for some $\lambda >0$ depending on $\Lambda_1, \Lambda_2, \|V\|_{\infty}$, that is
$$ - \divergence \left(A(x) \nabla \phi_{-} \right) + V(x) \phi_{-}\leq 0\  \text{in}\ \Omega,\qquad \phi_{-} \leq 1\ \text{on}\ \partial\Omega,$$
and
$$ - \divergence \left(A(x) \nabla \phi_{+} \right) + V(x) \phi_{+}\geq 0\  \text{in}\ \Omega,\qquad \phi_{+} \geq 1\ \text{on}\ \partial\Omega.$$
So by the weak maximum principle stated in \cite[Theorem 8.1]{GT01}, we have that
$$ \phi_{-}(x) \leq \phi(x) \leq \phi_{+}(x) \qquad \forall x \in \Omega.$$
In particular, this proves that there exists $c=c(\Omega,\Omega_0,\Lambda_1, \Lambda_2,\|V\|_{\infty})>0$ such that 
\begin{equation}
    \label{eq:boundphi}
    c \leq \phi(x) \leq 1\qquad \forall x \in \Omega.
\end{equation}
In particular, \eqref{eq:boundphi} implies \eqref{eq:lowerbound}.

Moreover, from local $W^{2,p}$-regularity estimate from \cite[Theorem 9.11]{GT01} and \eqref{eq:boundphi}, we have that for $1 < p < +\infty$, there exists $C=C(\Omega,\Omega_0,\Lambda_1, \Lambda_2,\|V\|_{\infty},p)>0$ such that
$$ \|\phi\|_{W^{2,p}(\Omega_0)} \leq C  \|\phi\|_{L^p(\Omega)} \leq C |\Omega|^{1/p} \|\phi\|_{L^{\infty}(\Omega)} \leq C |\Omega|^{1/p+1} .$$
By taking $p$ sufficiently large and by using Sobolev embeddings \cite[Theorem 7.26]{GT01} to guarantee that $W^{2,p}(\Omega_0) \hookrightarrow W^{1, \infty}(\Omega_0)$, we therefore deduce \eqref{eq:Winftyboundphi} from the previous estimate.
\end{proof}

We have the following result.
\begin{lemma}
\label{lem:reducediv}
Let $u \in W^{1,2}(\Omega)$ be a weak solution to 
\begin{equation}
    \label{eq:equationuOmega}
- \divergence \left(A(x) \nabla u \right) + V(x) u = 0\ \text{in}\ \Omega.
\end{equation}
Let $\phi$ be as in Lemma \ref{lem:auxiliary_fonction_for_elliptic_equation}. Then, $v = u/\phi \in W^{1,2}(\Omega_0)$ satisfies
\begin{equation}
\label{eq:equationv}
   -  \divergence \left(\phi^2 A \nabla v\right) = 0\ \text{in}\ \Omega_0.
\end{equation}
Moreover, the symmetric matrix $\hat{A} = \phi^2 A$ is uniformly elliptic and has Lipschitz entries, there exists $C_1=C(\Omega,\Omega_0,\Lambda_1, \Lambda_2,\|V\|_{\infty})>0$ and $C_2 =C(\Omega,\Omega_0,\Lambda_1, \Lambda_2,\|V\|_{\infty})>0$ such that
\begin{equation}
\label{eq:LipschitzAhatOmega}
    C_{1}^{-1} |\xi|^2 \leq \langle \hat{A}(x) \xi, \xi \rangle \leq   C_{1} |\xi|^2,\ |\hat{a}_{ij}(x) - \hat{a}_{ij}(y)| \leq  C_{2}\qquad \forall x, y \in \Omega_0,\ \forall \xi \in \R^d.
\end{equation}
\end{lemma}
\begin{proof}
The proof is a straightforward computation at the variational formulation level so we omit it.
\end{proof}

\subsection{Propagation of smallness}
\label{sec:prop_smallness}

In this part, we assume that the Lipschitz assumption on $A$, i.e. \eqref{eq:LipschitzAOmega} and the boundedness assumption on $V$, i.e. \eqref{eq:hypothesisVOmega} still hold together with $V \geq 0$.

We first deal with propagation of smallness for solutions to elliptic equations. We have the following quantitative propagation of smallness for solutions to divergence elliptic equations from \cite{LM18}. 

\begin{theorem}[{\cite[Theorem 2.1]{LM18}}]\label{thm:LM18_propagation_smallness_Solutions}
Let $\rho, m, \delta >0$ and $\mathcal K, E \subset \Omega$, be measurable subsets such that 
\begin{equation}
    \label{eq:distanceandhaussdorff}
    \text{dist}( \mathcal K, \partial \Omega) \geq \rho,\quad  \text{dist}(E, \partial \Omega) \geq \rho \quad \text{and} \quad  \mathcal C^{d-1+\delta}_{\mathcal H}(E) \geq m.
\end{equation}
There exist $C=C(\Omega,  \Lambda_1, \Lambda_2, \rho, m, \delta)>0$ and $ \alpha=\alpha(\Omega,  \Lambda_1, \Lambda_2, \rho, m, \delta) \in (0,1)$ such that for every weak solution $u \in W^{1,2}(\Omega) \cap L^{\infty}(\Omega)$ of the elliptic equation
\begin{equation}
\label{eq:elliptic_divergence_form}
    - \nabla \cdot (A(x) \nabla u)=0 \ \text{in} \ \Omega,
\end{equation}
we have
\begin{equation}
\label{eq:PropagationSmallness_Divform_Solutions}
  \sup_{\mathcal K} | u| \leq C (\sup_{E} | u|)^{\alpha} (\sup_{\Omega} | u|)^{1-\alpha}.
\end{equation}
\end{theorem}
As a consequence, the proof of Theorem \ref{cor:propagationsmallnessschrodinger} is as follows.
\begin{proof}[Proof of Theorem \ref{cor:propagationsmallnessschrodinger}]
From Lemma \ref{lem:auxiliary_fonction_for_elliptic_equation}, there exists $\phi \in W^{1,\infty}(\Omega_0)$ satisfying the elliptic equation \eqref{eq:equationmultiplier} and the lower bound \eqref{eq:lowerbound}. Then from \Cref{lem:reducediv} defining $v=u/\phi$, $v$ satisfies \eqref{eq:equationv}. As a consequence, one can apply Theorem \ref{thm:LM18_propagation_smallness_Solutions} to $v$ and to the Lipschitz diffusion matrix $\hat{A} = \phi^2 A$ that satisfies \eqref{eq:LipschitzAhatOmega} in $\Omega_0$, note that we use \eqref{eq:Omega0}. We deduce the propagation of smallness for $v$ that leads to the propagation of smallness for $u$ i.e. \eqref{eq:PropagationSmallness_Divform_SolutionsV} by using the lower bound \eqref{eq:lowerbound}, the $W^{1,\infty}$-bound \eqref{eq:Winftyboundphi} and again \eqref{eq:Omega0} ensuring that $$\|u\|_{L^{\infty}(\Omega_0)} \leq \|u\|_{L^{\infty}(\hat \Omega)}.$$
This concludes the proof.
\end{proof}

We now present results on propagation of smallness for gradient of solutions to elliptic equations.

\begin{theorem}[{\cite[Theorem 5.1]{LM18}}]\label{thm:LM18_propagation_smallness_Gradients}
There exists $\delta_d \in (0,1)$ depending only on the dimension $d$ such that the following holds. Let $\rho, m >0$, $\delta \in [0, \delta_d]$ and $\mathcal K, E \subset \Omega$, be measurable subsets such that 
\begin{equation}
    \label{eq:distanceandhaussdorffgradient}
    \text{dist}( \mathcal K, \partial \Omega) \geq \rho,\quad \text{dist}(E, \partial \Omega) \geq \rho \quad \text{and} \quad \mathcal C^{d-1-\delta}_{\mathcal H}(E) \geq m.
\end{equation}
There exist $C=C(\Omega,  \Lambda_1, \Lambda_2, \rho, m, \delta)>0$ and $ \alpha=\alpha(\Omega,  \Lambda_1, \Lambda_2, \rho, m, \delta) \in (0,1)$ such that for every weak solution $u \in W^{1,2}(\Omega) \cap L^{\infty}(\Omega)$ of the elliptic equation
\begin{equation}
\label{eq:elliptic_divergence_formgradient}
    - \nabla \cdot (A(x) \nabla u)=0 \ \text{in} \ \Omega,
\end{equation}
we have
\begin{equation}\label{eq:PropagationSmallness_Divform}
    \sup_{\mathcal K} |\nabla u| \leq C \left(\sup_{E} |\nabla u|\right)^{\alpha} \left(\sup_{\Omega} |\nabla u|\right)^{1-\alpha}.
\end{equation}
\end{theorem}

By now, we aim at studying elliptic equations in non-divergence form. Actually, we are only able to deal with some particular cases, that are \eqref{eq:EllipticEquation_NonDivForm} from Theorem \ref{cor:PropagationSmallness_withPotential}.

\begin{proof}[Proof of Theorem \ref{cor:PropagationSmallness_withPotential}]
Let $\phi = \phi(x)$ be as in Lemma \ref{lem:auxiliary_fonction_for_elliptic_equation}. Then $\hat{v} = u(x,t)/\phi(x)$ is a solution to 
\begin{equation}
	\label{eq:EllipticEquation_NonDivForDivForm}
		\left\{
			\begin{array}{ll}
				  - \nabla_x \cdot ( \phi^2 A(x) \nabla_x \hat{v})- \kappa(x) \partial_{tt} (\phi^2 \hat{v}) =0 & \text{ in }  \Omega_0 \times (-1,1), \\
				\hat{v}(x,0) = 0 & \text{ in } \Omega_0.
			\end{array}
		\right.
\end{equation} 
One can then apply Theorem~\ref{thm:LM18_propagation_smallness_Gradients} to $\hat{\Omega}_0 = \Omega_0 \times (-1,+1)$, $\hat{\mathcal K} = \mathcal K \times \{0\}$ and $\hat{E} = E \times \{0\}$. Note that 
\begin{equation}
    \label{eq:distanceandhaussdorffgradientVProof}
    \text{dist}( \hat{\mathcal K}, \partial \hat{\Omega}_0) \geq \rho/2, \quad \text{dist}(\hat{E}, \partial \hat{\Omega}_0) \geq \rho/2 \quad \text{and} \quad  \mathcal C^{d+1-1-\delta}_{\mathcal H}(\hat{E}) \geq m.
\end{equation}
We then have
$$ \sup_{x \in \mathcal K} |\partial_t \hat{v}(x,0)| \leq C \left(\sup_{x \in E} |\partial_t \hat{v}(x,0)|\right)^{\alpha} \left( \sup_{(x, t) \in \Omega_0 \times (-1,1)} |\nabla_{x,t} \hat{v}(x,t)|\right)^{1-\alpha}.$$
By using \eqref{eq:lowerbound}, the $W^{1, \infty}$-bound on $\phi$, i.e. \eqref{eq:Winftyboundphi} in $\Omega_0$ and also
$$\|\hat u\|_{W^{1,\infty}(\hat \Omega_0 \times (-1,1))},$$
we then deduce \eqref{eq:PropagationSmallness_withPotential}.
\end{proof}

\section{Proof of the spectral estimates}

This section aims at proving \Cref{thm:spectralcompact} and \Cref{thm:spectralRd}.

\subsection{Spectral estimates on compact manifolds}

The goal of this part is to prove Theorem \ref{thm:spectralcompact}. In the first part, we first reduce the obtaining of the spectral estimates for sets of positive Lebesgue measures \eqref{eq:measurablecompact} to the obtaining of spectral estimates for sets of positive Hausdorff measures \eqref{eq:hausdorffcompact}. In the second and third parts, we establish the results of Theorem~\ref{thm:spectralcompact} for manifold $M$ without boundary. Firstly, we prove a local version of \eqref{eq:hausdorffcompact} i.e. replacing the $L^{\infty}$-bound on $M$ in the left hand side of \eqref{eq:hausdorffcompact} to a $L^{\infty}$-bound on a chart of $M$. Secondly, by using the compactness and the connectedness of the manifold $M$, we propagate these local spectral estimates to the whole manifold. In the fourth part, we end the proof of Theorem~\ref{thm:spectralcompact} and prove the case when $\partial M \neq \emptyset$ with Dirichlet or Neumann boundary conditions on $\partial M$. The proof is an application of the double manifold trick introduced in \cite[Section 3]{BM23}.

\subsubsection{Reduction of spectral estimates to sets of positive Hausdorff measures}

In this part, we prove that the spectral estimates for sets of positive Hausdorff measures \eqref{eq:hausdorffcompact} imply the spectral estimates for sets of positive Lebesgue measures  \eqref{eq:measurablecompact}. 

Let $\omega \subset M$ such that $|\omega| > m > 0$. Let us define
$u = \Pi_{\Lambda }u $ with $\|u\|_{L^2(M)} =1$. Let us consider
\begin{equation}
    \label{eq:hatomega}
    \hat{\omega} = \{x \in \omega\ ;\ |u(x)| \leq \frac{1}{2C} e^{-C \sqrt{\Lambda}} \|u\|_{L^{\infty}(M)}\}.
\end{equation}
If $|\hat{\omega}| \geq m/2$, then we have for $\delta = \delta(d) \in (0,1)$, by applying \eqref{eq:equivalencehaussdorff},
$$ \mathcal{C}_{\mathcal{H}}^d(\hat{\omega})> c_d \frac m 2 \Rightarrow \mathcal{C}_\mathcal{H}^{d- \delta } (\hat{\omega}) \geq \min(1, c_d \frac m 2 ).$$
Then, one can apply \eqref{eq:hausdorffcompact} to $\hat{\omega}$ to get by definition of \eqref{eq:hatomega},
\begin{equation*}
\| u\|_{L^\infty(M)} \leq  C e^{C \sqrt{\Lambda}} \sup_{x \in \hat{\omega}} |u (x)| \leq \frac {\| u\|_{L^\infty(M)}} 2.
\end{equation*} 
This is impossible because this leads to $u=0$. Therefore $|\hat{\omega}| < m/2$ and consequently
$$ \int_{{\omega}} |u(x)| dx  \geq \int_{{\omega \setminus \hat{\omega}}} |u(x)| dx  \geq \frac{ m } { (4C) } e^{-C \sqrt{\Lambda}} \| u\|_{L^\infty(M)},
$$ leading to \eqref{eq:measurablecompact}. 

\subsubsection{Local spectral estimates}

In this part, we assume that $M$ is without boundary, $\partial M=\emptyset$. The purpose is to establish local spectral estimates holding in each charts of the manifold $M$. We recall that, in section~\ref{sec:Main_result_spectral_estimates_on_manifold}, we have fixed an atlas $\mathcal A= (\mathcal V_\sigma, \Psi_\sigma)_{\sigma \in \mathcal J}$ containing a finite number of charts with $W^{2, \infty}\cap C^1$-diffeomorphisms $\Psi_\sigma :  \mathcal V_\sigma \longrightarrow \Psi_\sigma(\mathcal V_{\sigma}) \subset \rr^{d-1}\times \rr_+$ such that there exists a family of open sets $(\mathcal U_\sigma)_{\sigma \in \mathcal J}$ satisfying 
\begin{equation}\label{eq:charts_recovering1}
M = \bigcup_{\sigma \in \mathcal J} \mathcal U_\sigma,
\end{equation}
and such that $ \mathcal U_\sigma$ is compactly included in the open set $  \mathcal V_\sigma$ in $M$, for all $\sigma \in \mathcal J$. Moreover, since $\partial M =\emptyset$, $\Psi_\sigma(\mathcal V_{\sigma})$ is an open set of $\rr^d$, for any $\sigma \in \mathcal J$.

The main result of this part is the following one.
\begin{proposition}\label{prop:local_spectral_estimates}
There exists $\delta_d\in (0,1)$ such that for all $\delta \in [0,\delta_d]$ and for every $\sigma \in \mathcal J$ and $m>0$, there exist  $C=C(M, g, V, \sigma,m,\delta)>0$ and $\alpha=\alpha(M,g, V, \sigma,m, \delta)\in (0,1)$ such that for all subsets $\omega$ with $C^{d-\delta}_{\mathcal H}(\omega \cap \mathcal U_\sigma)>m$ and $\Lambda > 0$,
\begin{equation}\label{eq:local_spectral_estimates}
\|\Pi_{\Lambda} u\|_{L^{\infty}( \mathcal U_\sigma)} \leq C  e^{C \sqrt {\Lambda}} \left( \sup_{\omega \cap \mathcal U_\sigma} |\Pi_{\Lambda} u|\right)^{\alpha} \|\Pi_{\Lambda} u\|^{1-\alpha}_{L^{\infty}(M)}\qquad \forall u \in L^2(M).
\end{equation}
\end{proposition}

\begin{proof}
First, one can assume that $V \geq 0$ just by considering the elliptic operator 
\begin{equation}
    \label{eq:HAvTranslated}
    H_{g,V} + \|V\|_{\infty} u\qquad \forall u \in \mathrm{Dom}(H_{g,V}),
\end{equation}
that has the same eigenfunctions $(\phi_k)_{k \geq 1}$ as the elliptic operator $H_{g,V}$ corresponding to the shifted eigenvalues $\lambda_k +  \|V\|_{\infty}$.

We fix $\sigma \in \mathcal J$, we now work in a coordinate patch $ \mathcal U_\sigma \subset \subset  \mathcal V_\sigma$ and we define the sets
\begin{equation}
    \mathcal V = \Psi_\sigma( \mathcal V_\sigma),\ \mathcal U = \Psi_\sigma( \mathcal U_\sigma)\ \text{and}\ E = \Psi_\sigma(\omega \cap  \mathcal U_\sigma).
\end{equation}

For $\Lambda >0$, we then consider
\begin{equation}
    u(x) =  \sum_{\lambda_k \leq \Lambda} u_k \phi_k  (x)\qquad x \in  M,
\end{equation}
and its local push forward version $U$
\begin{equation}
    U(x) = u \circ \Psi_\sigma^{-1}(x) = \sum_{\lambda_k \leq \Lambda} u_k (\phi_k \circ \Psi_\sigma^{-1}) (x) = \sum_{\lambda_k \leq \Lambda} u_k \Phi_k (x)\qquad x \in  \mathcal V.
\end{equation}
We then add an extra-variable to $u$ by defining
\begin{equation}
    \hat{u}(x,t) =  \sum_{\lambda_k \leq \Lambda} u_k \frac{\sinh(\sqrt{\lambda_k }t )}{\sqrt{\lambda_k}} \phi_k(x)\qquad (x,t) \in  M \times (-2,+2),
\end{equation}
and its local push forward version $\hat{U}$
\begin{equation}
    \hat{U}(x,t) = \sum_{\lambda_k \leq \Lambda} u_k \frac{\sinh(\sqrt{\lambda_k }t)}{\sqrt{\lambda_k}} \Phi_k(x)\qquad (x,t) \in  \mathcal V \times (-2,+2).
\end{equation}

In the chart $(\mathcal V_{\sigma}, \Psi_{\sigma})$, let us consider $(g_{i,j})_{1 \leq i,j \leq d}$ the local coordinates of the metric $g$. We define for $x \in \mathcal V$, $G(x)= (g_{i,j}(\Psi_{\sigma}^{-1}(x)))_{1 \leq i,j \leq d}$. We observe that $\hat{U}$ solves 
\begin{equation}
	\label{eq:EllipticEquation_NonDivFormProof}
		\left\{
			\begin{array}{ll}
				  - \nabla_x \cdot (A(x) \nabla_x \hat{U}) - \kappa(x) \partial_{tt} \hat{U}+ \hat{V}(x) \hat{U}=0 & \text{ in }  \mathcal V \times (-2,2), \\
				\hat{U}(x,0) = 0 & \text{ in } \mathcal V,
			\end{array}
		\right.
\end{equation}
with $A=G(x)^{-1} \sqrt{\det G(x)}$, $\kappa(x)=\sqrt{\det G(x)}$ and $\hat{V}(x)=\sqrt{\det G(x)} V\left(\Psi_{\sigma}^{-1}(x)\right)$ satisfying the hypotheses \eqref{eq:LipschitzAOmega}, \eqref{eq:hypothesisVOmega} and \eqref{eq:LipschitzkappaOmega}. We can then apply Corollary \ref{cor:PropagationSmallness_withPotential} to $\hat{U}$ with $\Omega = \tilde{\mathcal V}$, $\mathcal K = \mathcal U$ and $E = \Psi_\sigma(\omega \cap  \mathcal U_\sigma)$ such that $\mathcal U \subset \subset \tilde{\mathcal V} \subset \subset \mathcal V$ to get
\begin{equation}
\label{eq:PropagationSmallness_withPotentialProof}
    \sup_{x \in \mathcal U} |\partial_t \hat{U}(x,0)| \leq C \left(\sup_{x \in E} |\partial_t \hat{U}(x,0)|\right)^\alpha \left\|\hat U\right\|^{1-\alpha}_{W_{t,x}^{1, \infty}(\Omega\times(-1,1))}.
\end{equation}
The left hand side of \eqref{eq:PropagationSmallness_withPotentialProof} exactly gives
\begin{equation}
\label{eq:lefthandside}
     \sup_{x \in \mathcal U} |\partial_t \hat{U}(x,0)| =  \sup_{x \in \mathcal U} |U(x)| = \sup_{x \in \mathcal U_\sigma} |u(x)|.
\end{equation}
The first right hand side term of \eqref{eq:PropagationSmallness_withPotentialProof} exactly gives
\begin{equation}
\label{eq:righthandside1}
    \sup_{x \in E} |\partial_t \hat{U}(x,0)| = \sup_{x \in E} |U(x)|  = \sup_{x \in \omega \cap \mathcal U_\sigma} |u(x)|.
\end{equation}
Moreover, by elliptic regularity, see \cite[Theorem 9.11]{GT01}, Sobolev embeddings and a bootstrap argument, using in particular the elliptic equation \eqref{eq:EllipticEquation_NonDivFormProof}, we obtain that the second right hand side term of \eqref{eq:PropagationSmallness_withPotentialProof} is bounded as follows
\begin{equation}
\label{eq:righthandside2}
   \|\hat U\|_{W_{t,x}^{1, \infty}(\Omega\times(-1,1))} \leq C \|\hat{U}\|_{L^2(\mathcal{V}\times(-2,+2))}.
\end{equation}
By using a change of variable, we then obtain that
\begin{equation}
\label{eq:righthandside21}
    \|\hat{U}\|_{L^2(\mathcal{V}\times(-2,+2))} \leq C     \|\hat{u}\|_{L^2(\mathcal{V}_\sigma\times(-2,+2))} \leq \|\hat{u}\|_{L^2(M\times(-2,+2))}.
\end{equation}
Now, by using the orthogonality of the eigenfunctions $(\phi_k)_{k \geq 1}$ in $L^2(M)$, we then obtain that
\begin{equation}
\label{eq:righthandside22}
    \|\hat{u}\|_{L^2(M\times(-2,+2))} \leq C \exp(C \sqrt{\Lambda}) \|u\|_{L^2(M)}.
\end{equation}
We now gather \eqref{eq:PropagationSmallness_withPotentialProof}, \eqref{eq:lefthandside}, \eqref{eq:righthandside1}, \eqref{eq:righthandside2}, \eqref{eq:righthandside21}, \eqref{eq:righthandside22} to get 
\eqref{eq:local_spectral_estimates}.
\end{proof}

\subsubsection{Propagation to the whole manifold}

In this part, we prove Theorem \ref{thm:spectralcompact} by using the connectedness of the manifold $M$ to propagate the estimates \eqref{eq:local_spectral_estimates} to the whole manifold $M$, which is still assumed to be without boundary.

\begin{proof}[Proof of Theorem \ref{thm:spectralcompact} in the case $\partial M=\emptyset$]
We define the following subset $\mathcal I \subset \mathcal J$ such that $\sigma \in \mathcal I$ if and only if there exist $C_\sigma>0$ and $\alpha_\sigma \in (0,1)$ so that 
\begin{equation}
 \|\Pi_{\Lambda} u\|_{L^{\infty}(\mathcal U_\sigma)} \leq C_\sigma e^{C_\sigma \sqrt{\Lambda}} (\sup_{x \in \omega} |\Pi_{\Lambda}u(x)|)^{\alpha_\sigma} \|\Pi_{\Lambda} u\|^{1-\alpha_\sigma}_{L^{\infty}(M)} \quad \forall u \in L^2(M),\ \forall \Lambda > 0.
\end{equation}

Thanks to \eqref{eq:charts_recovering1}, we have $$M =\bigcup_{\sigma \in \mathcal I} \mathcal U_\sigma \cup \bigcup_{\sigma \notin \mathcal I} \mathcal U_\sigma.$$
First of all, $\mathcal I$ is not empty. Indeed, since $C^{d-\delta'}_{\mathcal H}(\omega)>m$, there exists $\sigma_0 \in \mathcal J$ such that $C^{d-\delta'}_{\mathcal H}(\omega \cap \mathcal U_{\sigma_0})>\frac mN$, where $N$ denotes the cardinality of the finite set $\mathcal J$. It is then sufficient to apply Proposition~\ref{prop:local_spectral_estimates} to obtain $\sigma_0 \in \mathcal I$.

Let us assume by contradiction that $\mathcal I \neq \mathcal J$. Since $M$ is connected, there exist $\sigma \in \mathcal I$ and $\tilde \sigma \notin \mathcal I$ such that $\mathcal U_\sigma \cap  \mathcal U_{\tilde \sigma} \neq \emptyset$. By applying Proposition~\ref{prop:local_spectral_estimates} with $j=\tilde{\sigma}$ and $\omega =\mathcal U_{\sigma} \cap \mathcal U_{\tilde \sigma}$ that is open, there exist $C_{\sigma,\tilde \sigma}>0$ and $0< \alpha_{\sigma,\tilde \sigma}<1$ such that for $\Lambda >0$ and $u \in L^2(M)$,
\begin{align}
\|\Pi_{\Lambda} u\|_{L^{\infty}(\mathcal U_{\tilde \sigma})} & \leq C_{\sigma, \tilde{\sigma}} e^{C_{\sigma,\tilde \sigma} \sqrt{\Lambda}}  \|\Pi_{\Lambda}u\|_{L^{\infty}(\mathcal U_{\sigma} \cap \mathcal U_{\tilde{\sigma}})}^{\alpha_{\sigma,\tilde \sigma}} \|\Pi_{\Lambda} u\|^{1-\alpha_{\sigma,\tilde \sigma}}_{L^{\infty}(M)},\notag
\end{align}
so that
\begin{align}
\|\Pi_{\Lambda} u\|_{L^{\infty}(\mathcal U_{\tilde \sigma})} & \leq C_{\sigma,\tilde \sigma} e^{C_{\sigma,\tilde \sigma} \sqrt{\Lambda}} \|\Pi_{\Lambda}u\|_{L^{\infty}(\mathcal U_{\sigma} )}^{\alpha_{\sigma,\tilde \sigma}} \|\Pi_{\Lambda} u\|^{1-\alpha_{\sigma,\tilde \sigma}}_{L^{\infty}(M)}.\label{eq:local_estimate_for_k}
\end{align}

Moreover, since $\sigma \in \mathcal I$, there exist $C'_{\sigma}>0$ and $0<\alpha'_{\sigma}<1$ such that
\begin{equation}\label{eq:local_estimate_for_i}
\|\Pi_{\Lambda} u\|_{L^{\infty}(\mathcal U_\sigma)}  \leq C'_{\sigma}  e^{C'_\sigma \sqrt{\Lambda}}( \sup_{x \in \omega} |\Pi_{\Lambda}u(x)|)^{\alpha'_\sigma} \|\Pi_{\Lambda} u\|^{1-\alpha'_\sigma}_{L^{\infty}(M)}\quad \forall u \in L^2(M),\ \forall \Lambda > 0.
\end{equation}

Let $\Lambda>0$ and $u \in L^2(M)$ such that $\|\Pi_{\Lambda}u\|_{L^{\infty}(M)}=1$. 
We therefore deduce from \eqref{eq:local_estimate_for_k} and \eqref{eq:local_estimate_for_i} that

\begin{equation*}
\|\Pi_{\Lambda} u\|_{L^{\infty}(\mathcal U_{\tilde \sigma})}  \leq C''_{\sigma, \tilde \sigma} e^{C''_{\sigma, \tilde \sigma} \sqrt{\Lambda}}( \sup_{x \in \omega} |\Pi_{\Lambda}u(x)|)^{\beta_{\sigma, \tilde \sigma}}
\end{equation*}
with $0< \beta_{\sigma,\tilde \sigma}=\alpha'_\sigma \alpha_{\sigma,\tilde \sigma}<1$ and $C''_{\sigma,\tilde \sigma}= \max({C'}^{\alpha_{\sigma,\tilde \sigma}}_{\sigma} C_{\sigma,\tilde \sigma},\alpha_{\sigma,\tilde \sigma}C'_\sigma+C_{\sigma,\tilde \sigma})$. It readily follows that for all $\Lambda>0$ and for all $u \in L^2(M)$,

\begin{equation*}
\|\Pi_{\Lambda} u\|_{L^{\infty}(U_k)}  \leq C''_{\sigma,\tilde \sigma} e^{C''_{\sigma,\tilde \sigma} \sqrt{\Lambda}} (\sup_{x \in \omega} |\Pi_{\Lambda}u(x)|)^{\beta_{\sigma,\tilde \sigma}} \|\Pi_{\Lambda}u \|^{1-\beta_{\sigma,\tilde \sigma}}_{L^{\infty}(M)}.
\end{equation*}

Thus, $\tilde \sigma \in \mathcal I$ and this provides a contradiction. 

To conclude this subsection, we have $\mathcal I=\mathcal J$ and by defining
$$0<\alpha = \min_{\sigma \in \mathcal J}\alpha_\sigma<1 \quad \text{and} \quad C=\max_{\sigma \in \mathcal J}C_\sigma >0,$$
we have 
\begin{equation*}
 \|\Pi_{\Lambda}u\|_{L^{\infty}(M)} \leq C e^{C \sqrt{\Lambda}} (\sup_{x \in \omega} |\Pi_{\Lambda}u(x)|)^{\alpha} \|\Pi_{\Lambda}u\|^{1-\alpha}_{L^{\infty}(M)}\quad \forall u \in L^2(M),\ \forall \Lambda > 0,
\end{equation*}
which readily provides 
\begin{equation*}
 \|\Pi_{\Lambda}u\|_{L^{\infty}(M)} \leq C^{\frac{1}{\alpha}} e^{\frac C{\alpha} \sqrt{\Lambda}} \sup_{x \in \omega} |\Pi_{\Lambda}u(x)|\quad \forall u \in L^2(M),\ \forall \Lambda > 0.
\end{equation*}
This concludes the proof.
\end{proof}

\subsubsection{The double manifold}

In this part, we prove Theorem \ref{thm:spectralcompact} for a manifold with boundary $M$ and Dirichlet or Neumann boundary conditions on $\partial M$. The idea consists in reducing this question to the case of a manifold without boundary by gluing two copies of $M$ along the boundary in such a way the new double manifold $\widetilde{M}$ inherits a Lipschitz metric, which allows to apply the previous results (without boundary) to this double manifold. This is done in \cite[Section 3]{BM23} but the only point that we need to check in our setting is the equation satisfied by the eigenfunctions on the double manifold.\medskip

Let $\widetilde{M} = \overline{M} \times \{-1,1\}/\partial M$  the double space made of two copies of $\overline{M}$ where we identified the points on the boundary, $(x,-1)$ and $(x,1)$, $x\in \partial M$. 
\begin{theorem}[The double manifold]
\label{thm:double}
There exist a $C^{ \infty}$ structure on the double manifold $\widetilde{M}$, a metric $\widetilde{g} \in W^{1, \infty}$ on $\widetilde{M}$, a potential $\widetilde{V} \in L^{\infty}(\widetilde{M})$ such that the following holds.
\begin{itemize}
\item The maps
$$ i^\pm :  x\in M  \rightarrow (x, \pm 1) \in \widetilde{M} = M \times \{\pm 1\} / \partial M$$ are isometric embeddings.
\item The potential $\widetilde{V}$ is such that
$$ \widetilde{V}(x,\pm 1) = V(x)\qquad x \in M.$$
\item For any eigenfunction $\phi_{\lambda}$ with eigenvalue $\lambda$ of the operator $H_{g,V}$ with Dirichlet or Neumann boundary conditions, there exists an eigenfunction $\widetilde{\phi_{\lambda}}$ with the same eigenvalue $\lambda$ of the operator $H_{\tilde{g}, \widetilde{V}}$ on $\widetilde{M}$ such that 
\begin{equation}
\label{eq:ext}
\widetilde{\phi_{\lambda}} \mid_{M \times \{1\}} = \phi_{\lambda}, \quad \widetilde{\phi_{\lambda}} \mid_{M \times \{-1\}} = \begin{cases}  -\phi_{\lambda} \quad &(\text{Dirichlet boundary conditions}),\\
 \phi_{\lambda} \qquad &(\text{Neumann boundary conditions}). \end{cases}
\end{equation}
\end{itemize}
\end{theorem}

The proof exactly follows the same lines as the one of \cite[Theorem 7]{BM23}. The main difference comes from the fact that we need to deal with a potential $V \in L^{\infty}$. One of the main difficulty in the proof of \cite[Theorem 7]{BM23} consists in computing the Laplacian of $\tilde{\phi}_{\lambda}$ on the new manifold $\tilde{M}$, thanks to the jump formula. In particular, there is no new difficulty to add this potential in the proof of \cite[Theorem 7]{BM23}. For the sake of conciseness, we omit the proof of Theorem~\ref{thm:double}.

The results of Theorem~\ref{thm:spectralcompact} are readily implied by Theorem~\ref{thm:double}.

\subsection{Spectral estimates on the Euclidean space}

The goal of this part is to prove Theorem \ref{thm:spectralRd}. In the first part, we first show that \eqref{eq:measurableRd} can be deduced from \eqref{eq:hausdorffRd} in the case when $\delta=0$. In the second part, we prove the spectral estimates \eqref{eq:hausdorffRd}.

\subsubsection{Reduction of spectral estimates to uniformly distributed sets}

In this first part, we explain how \eqref{eq:measurableRd} can be deduced from \eqref{eq:hausdorffRd}. Let $g \in L^2(\rr^d)$ and $\Lambda > 0$ such that $g=\Pi_{\Lambda} g$. If $\omega \subset \rr^d$ is a thick subset, i.e. satisfying \eqref{eq:measurethick} we define an auxiliary subset $\tilde{\omega}=\bigcup_{k \in \mathbb Z^d} \tilde{\omega}_k$, where for all $k \in \mathbb Z^d$, 
$$\tilde{\omega}_k = \left\{ x \in \omega \cap B(k,R), \quad |g(x)|^2 \leq \frac{2}{|\omega \cap B(k,R)|} \int_{\omega \cap B(k,R)} |g(y)|^2 dy \right\} \subset B(k,R).$$
By definition, we have for all $k \in \mathbb Z^d$,
$$\int_{\omega \cap B(k,R)} |g(x)|^2 dx \geq \int_{(\omega\cap B(k,R) )\setminus \tilde{\omega}_k} |g(x)|^2 dx \geq \frac{2 |(\omega\cap B(k,R) )\setminus \tilde{\omega}_k|}{|\omega \cap B(k,R)|} \int_{\omega \cap B(k,R)} |g(y)|^2dy.$$
Thus, if $$\int_{\omega \cap B(k,R)} |g(y)|^2dy > 0,$$
then $$|(\omega \cap B(k,R))\setminus \tilde{\omega}_k| \leq \frac{|\omega\cap B(k,R)|}2,$$
which implies $$|\tilde{\omega}_k| \geq \frac{\gamma}2 |B(k,R)|,$$
thanks to the thickness property satisfied by $\omega$ i.e. \eqref{eq:measurethick}. Otherwise, if $$\int_{\omega \cap B(k,R)} |g(y)|^2d y = 0,$$
then $g \equiv 0$ in $\omega \cap B(k,R)$ so $\tilde{\omega}_k =\omega \cap B(k,R)$, therefore 
$$|\tilde{\omega}_k|=|\omega \cap B(k,R)| \geq \gamma |B(k,R)| >\frac{\gamma}2 |B(k,R)|.$$
Finally, $\tilde{\omega}$ is still a thick subset of $\rr^d$, and it follows from the spectral estimate \eqref{eq:hausdorffRd} that 

\begin{align*}
\left\| g \right\|^2_{L^2(\rr^d)}  & \leq C e^{C \sqrt{\Lambda}}\sum_{k \in \mathbb Z^d} \sup_{x \in \tilde{\omega}_k}| g(x)|^2 \\
& \leq C e^{C \sqrt \Lambda} \sum_{k \in \mathbb Z^d} \frac{2}{|\omega \cap B(k,R)|} \int_{\omega \cap B(k,R)} |g(x)|^2dx\\
& \leq \frac 2{\gamma |B(0,R)|} C e^{C \sqrt \Lambda} \sum_{k \in \mathbb Z^d}  \int_{\omega \cap B(k,R)} |g(x)|^2dx\\
& \leq C e^{C \sqrt \Lambda} \int_{\omega} |g(x)|^2dx,
\end{align*}
since $$1 \leq \sum_{k \in \mathbb Z^d} \un_{B(k,R)} \leq C(d).$$
This concludes the proof of \eqref{eq:measurableRd}.

\subsubsection{Spectral estimates}

In this part, we prove the spectral estimates \eqref{eq:hausdorffRd}. We can assume $V\geq 0$ since spectral estimates for $H_{g, V, \kappa}+ \|V\|_{L^{\infty}}$ readily imply spectral estimates for $H_{g, V, \kappa}$. Let $m, R >0$ and $\omega \subset \rr^d$ satisfying 
$$\forall x \in \rr^d, \quad C^{d-\delta}_{\mathcal H}(\omega \cap B(x,R)) \geq m,$$
for $\delta \in [0,\delta_d]$ with $0<\delta_d<1$ provided by Theorem~\ref{cor:PropagationSmallness_withPotential}.

Let us fix $\lambda > 0$ and $f = \Pi_{\lambda} f$. The strategy, inspired by the works \cite{BM21,BM23}, consists in adding a ghost dimension and defining the following $(d+1)$-dimensional function
\begin{equation}\label{eq:ghost_dimension_function}
    F_{\lambda}(x,y)= \frac{\sinh\left(\sqrt{H_{g, V, \kappa}} y \right)}{\sqrt{H_{g, V, \kappa}}} \Pi_{\lambda} f(x), \quad (x,y) \in \rr^d \times (-1,1).
\end{equation}

Notice that $F_{\lambda} \in H^2(\rr^d \times (-5R,5R))$ and satisfies the following elliptic equation
\begin{equation}
    \label{eq:equationFLambda}
    - \nabla_x \cdot(\kappa(x) g^{-1}(x) \nabla_x F_{\lambda}) -\kappa(x) \partial_y^2 F_{\lambda} + \kappa(x) V(x)F_{\lambda}=0 \quad \text{in} \quad \rr^d\times (-5R,5R).
\end{equation}
Moreover, we have 
$$F_{\lambda}(\cdot, 0)=0 \quad \text{and} \quad \partial_y F(\cdot, 0) = \Pi_{\lambda}f \quad \text{on} \quad \rr^d.$$

In the following, the constants will be of the form $C=C(M, \Lambda_1, \Lambda_2, \|V\|_{\infty}, R, m, \delta)>0$ and can change from one line to another.

Since the conclusion of Corollary~\ref{cor:PropagationSmallness_withPotential} is invariant by translations and because of the uniform bounds on $g$ and $V$ and $\kappa$, we have that there exist positive constants $C>0$ and $0<\alpha<1$ such that for all $k \in \mathbb Z^d$,
\begin{equation*}
    \sup_{x \in B(k,R)} |\partial_y F_{\lambda}(x,0)| \leq C (\sup_{x \in \omega\cap B(k,R)} |\partial_y F_{\lambda}(x,0)|)^\alpha \|F_{\lambda}(x,y)\|_{W_{x,y}^{1, \infty}( B(k,2R)\times (-R,R))}^{1-\alpha},
\end{equation*}
which implies 
\begin{equation}\label{eq:Propagation_Smallness_FrequencyCutOff}
    \sup_{x \in B(k,R)} |\Pi_{\lambda}f(x)| \leq C (\sup_{x \in \omega\cap B(k,R)} |\Pi_{\lambda}f(x)|)^\alpha \|F_{\lambda}(x,y)\|_{W_{x,y}^{1, \infty}( B(k,2R)\times (-R,R))}^{1-\alpha}.
\end{equation}
By elliptic regularity, see \cite[Theorem 9.11]{GT01}, Sobolev embeddings and a bootstrap argument, using in particular the elliptic equation \eqref{eq:equationFLambda} there exists a positive constant $C>0$ such that 
\begin{align*}\|F_{\lambda}(x,y)\|_{W_{x,y}^{1, \infty}( B(k,2R)\times (-R,R))}
& \leq C \|F_{\lambda}\|_{L^2(B_{d+1}((k,0),5R))}.
\end{align*}
This implies, together with \eqref{eq:Propagation_Smallness_FrequencyCutOff}, that 
\begin{align}
    \|\Pi_{\lambda} f\|^2_{L^2(B(k,R))} & \leq \sup_{x \in B(k,R)} |\Pi_{\lambda}f(x)|^2\notag
    \\
    & \leq C(\sup_{x \in \omega \cap B(k,R)}|\Pi_{\lambda}f(x)|)^{2\alpha} \|F_{\lambda}\|^{2(1-\alpha)}_{L^2(B_{d+1}((k,0),5R))}.
\end{align}
It therefore follows from Young's inequality that there exists $\beta>0$ such that for all $\varepsilon>0$,
\begin{equation*}
    \|\Pi_{\lambda} f\|^2_{L^2(B(k,R))} \leq C \varepsilon^{-\beta} \sup_{x \in \omega \cap B(k,R)}|\Pi_{\lambda}f(x)|^2 +\varepsilon \|F_\lambda\|^2_{L^2(B_{d+1}((k,0),5R))}.
\end{equation*}
By summing over all the integers $k \in \mathbb Z^d$ and using the facts that
$$1 \leq \sum_{k \in \mathbb Z^d} \un_{B(k,R)} \quad \text{and} \quad \sum_{k \in \mathbb Z^d} \un_{B_{d+1}((k,0), 5R)} \leq C(d) \un_{\rr^d \times (-5R, 5R)},$$
for some positive constant $C(d)\geq 1$ depending only on the dimension,
we have for all $\varepsilon>0$,
\begin{multline}\label{eq:Approximative_SpectralEstimate_Euclidean}
    \|\Pi_{\lambda}f\|^2_{L^2(\rr^d)} \leq \sum_{k \in \mathbb Z^d} \|\Pi_{\lambda} f \|^2_{L^2(B(k,R))}\\
    \leq C \varepsilon^{-\beta} \sum_{k \in \mathbb Z^d} \sup_{x \in \omega \cap B(k,R)}|\Pi_{\lambda}f(x)|^2 + \varepsilon C(d) \|F_{\lambda}\|^2_{L^2(\rr^d \times(-5R,5R))}.
\end{multline}
By now, let us show that there exists a positive constant $C>0$ such that 
\begin{equation}\label{eq:H1estimation_FrequencyCutOff}
\|F_{\lambda} \|^2_{L^2(\rr^d \times (-5R, 5R))} \leq C e^{C\sqrt \lambda} \|\Pi_{\lambda} f\|^2_{L^2(\rr^d)}.
\end{equation}
Indeed,  by using the fact that there exists a positive constant $C>0$ such that 
$$ \frac{\sinh(yt)^2}{t^2} \leq Ce^{C \sqrt \lambda}\qquad\forall y \in (-5R, 5R), \forall t \in (0, \sqrt \lambda),$$ we have 
\begin{align*}  & \|F_\lambda\|^2_{L^2(\rr^d \times (-5R, 5R))} &\leq \int_{-5R}^{5R} \int_{\rr^d} \left|\frac{\sinh\left(y \sqrt{H_{g, V, \kappa}}\right)}{\sqrt{H_{ g, V, \kappa}}} \Pi_{\lambda} f(x)\right|^2 \leq 10 R Ce^{C \sqrt \lambda} \|\Pi_{\lambda} f\|^2_{L^2(\rr^d)}.
\end{align*}

Finally, we have shown that \eqref{eq:H1estimation_FrequencyCutOff} holds for some constant $C>0$. It follows from \eqref{eq:Approximative_SpectralEstimate_Euclidean} with $\varepsilon= \frac{1}{C} e^{-C \sqrt \lambda}$ and \eqref{eq:H1estimation_FrequencyCutOff} that 
\begin{equation*}
    \|\Pi_{\lambda} f\|^2_{L^2(\rr^d)} \leq C e^{C \sqrt \lambda} \sum_{k \in \mathbb Z^d} \sup_{x \in \omega \cap B(k,R)} |\Pi_{\lambda} f(x)|^2.
\end{equation*}
 This ends the proof of Theorem \ref{thm:spectralRd}.

\bibliographystyle{alpha}
\small{\bibliography{SpectralEstimates}}

\end{document}